\documentclass{amsart}
\usepackage{amsfonts}
\usepackage{hyperref}

\newtheorem{theorem}{Theorem}
\theoremstyle{plain}

\newtheorem{conjecture}{Conjecture}
\newtheorem{corollary}{Corollary}

\newtheorem{lemma}{Lemma}

\newtheorem{proposition}{Proposition}
\newtheorem{remark}{Remark}

\numberwithin{equation}{section}
\input{tcilatex}

\begin{document}
\title[Accurate approximations for the elliptic integral of the second kind]{%
Very accurate approximations for the elliptic integrals of the second kind
in terms of Stolarsky means}
\author{Zhen-Hang Yang}
\email{yzhkm@163.com}
\address{Customer Service Center, ZPEPC Electric Power Research Institute,
Hangzhou, Zhejiang, China, 310009}
\date{May 10, 2015}
\subjclass[2010]{Primary 33E05, 26D15; Secondary 26E60, 26A48}
\keywords{Complete elliptic integrals for the second kind, Stolarsky means,
monotonicity, inequality}
\dedicatory{Dedicated to the my mother Ru-Yi Jiang}
\thanks{This paper is in final form and no version of it will be submitted
for publication elsewhere.}

\begin{abstract}
For $a,b>0$ with $a\neq b$, the Stolarsky means are defined by%
\begin{equation*}
S_{p,q}\left( a,b\right) =\left( {\dfrac{q(a^{p}-b^{p})}{p(a^{q}-b^{q})}}%
\right) ^{1/(p-q)}\text{ if }pq\left( p-q\right) \neq 0
\end{equation*}%
and $S_{p,q}\left( a,b\right) $ is defined as its limits at $p=0$ or $q=0$
or $p=q$ if $pq\left( p-q\right) =0$. The complete elliptic integrals of the
second kind $E$ is defined on $\left( 0,1\right) $ by%
\begin{equation*}
E\left( r\right) =\int_{0}^{\pi /2}\sqrt{1-r^{2}\sin ^{2}t}dt.
\end{equation*}%
We prove that the functions%
\begin{equation*}
F\left( r\right) =\frac{1-\left( 2/\pi \right) E\left( r\right) }{%
1-S_{11/4,7/4}\left( 1,r^{\prime }\right) }\text{ and }G\left( r\right) =%
\frac{1-\left( 2/\pi \right) E\left( r\right) }{1-S_{5/2,2}\left(
1,r^{\prime }\right) }
\end{equation*}%
are strictly decreasing and increasing on $\left( 0,1\right) $,
respectively, where $r^{\prime }=\sqrt{1-r^{2}}$. These yield some very
accurate approximations for the complete elliptic integrals of the second
kind, which greatly improve some known results.
\end{abstract}

\maketitle

\section{Introduction}

For $r\in \left( 0,1\right) $, the well-known complete elliptic integrals of
the first and second kinds \cite{Bowman-IEFA-NY-1961}, \cite%
{Byrd-HEIES-NY-1971} are defined by%
\begin{equation}
K\left( r\right) =\int_{0}^{\pi /2}\frac{dt}{\sqrt{1-r^{2}\sin ^{2}t}}\text{%
, \ }K\left( 0^{+}\right) =\frac{\pi }{2}\text{, \ }K\left( 1^{-}\right)
=\infty  \label{K}
\end{equation}%
and%
\begin{equation}
E\left( r\right) =\int_{0}^{\pi /2}\sqrt{1-r^{2}\sin ^{2}t}dt\text{, \ }%
E\left( 0^{+}\right) =\frac{\pi }{2}\text{, \ }E\left( 1^{-}\right) =1,
\label{E}
\end{equation}%
respectively. These integrals can be expressed exactly in terms of Gaussian
hypergeometric function%
\begin{equation*}
_{2}F_{1}(a,b;c;z)=\sum_{n=0}^{\infty }\frac{\left( a\right) _{n}\left(
b\right) _{n}\left( a_{3}\right) _{n}}{\left( c\right) _{n}}\frac{z^{n}}{n!}%
,|z|<1,
\end{equation*}%
where $a,b,c\in \mathbb{R}$ with $c\neq 0,-1,-2,...$, $\left( a\right) _{n}$
is defined by $\left( a\right) _{0}=1$ for $a\neq 0$ and%
\begin{equation*}
\left( a\right) _{n}=a\left( a+1\right) \cdot \cdot \cdot \left(
a+n-1\right) ,a\neq 0.
\end{equation*}%
Indeed, we have%
\begin{eqnarray}
K &=&K\left( r\right) =\frac{\pi }{2}F\left( \frac{1}{2},\frac{1}{2}%
;1;r^{2}\right) =\frac{\pi }{2}\sum_{n=0}^{\infty }\frac{\left( \tfrac{1}{2}%
\right) _{n}^{2}}{n!n!}r^{2n},  \label{Ke} \\
E &=&E\left( r\right) =\frac{\pi }{2}F\left( -\frac{1}{2},\frac{1}{2}%
;1;r^{2}\right) =\frac{\pi }{2}\sum_{n=0}^{\infty }\frac{\left( -\frac{1}{2}%
\right) _{n}\left( \frac{1}{2}\right) _{n}}{n!n!}r^{2n}.  \label{Ee}
\end{eqnarray}

It is well known that the complete elliptic integrals have many important
applications in physics, engineering, geometric function theory,
quasiconformal analysis, theory of mean values, number theory and other
related fields \cite{Anderson-CA-14-1998}, \cite{Anderson-PJM-192-2000}, 
\cite{Anderson-IJM-62-1988}, \cite{Baricz-LNM-1994}, \cite{Baricz-EM-2011}, 
\cite{Barnard-JMA-31-2000}, \cite{Barnard-JMA-32-2000}, \cite%
{Barnard-JMAA-260-2001}, \cite{Qiu-ACS-43-2000}, \cite{Qiu-JIA-4-1999}, \cite%
{Vuorinen-SM-121-1996}.

Let $l(1,r)$ be the arc length of an ellipse with semiaxis $1$ and $r\in
(0,1)$. Then%
\begin{equation*}
l\left( 1,r\right) =4E\left( r^{\prime }\right) ,
\end{equation*}%
where and in what follows $r^{\prime }=\sqrt{1-r^{2}}$. In 1883, Muir \cite%
{Muir-MM-12-1883} presented a simple approximation for $l\left( 1,r\right) $
by $2\pi A_{3/2}\left( 1,r\right) $, where%
\begin{equation*}
A_{p}\left( a,b\right) =\left( \frac{a^{p}+b^{p}}{2}\right) ^{1/p}\text{ if }%
p\neq 0\text{ and }A_{0}\left( a,b\right) =\sqrt{ab}
\end{equation*}%
is the classical power mean of positive numbers $a$ and $b$. This mean
contains some simple ones, such as $A_{-1}\left( a,b\right) =H\left(
a,b\right) $ --harmonic mean, $A_{1}\left( a,b\right) =A\left( a,b\right) $
--arithmetic mean, $A_{0}\left( a,b\right) =G\left( a,b\right) $ --geometric
mean, and $A_{2}\left( a,b\right) =S\left( a,b\right) $ --root-square mean,
etc.

In 1996, Vuorinen \cite{Vuorinen-PFDE-1988} conjectured that the following
inequality%
\begin{equation}
E\left( r\right) \geq \frac{\pi }{2}A_{3/2}\left( 1,r^{\prime }\right) =:%
\frac{\pi }{2}\mathcal{A}_{1}\left( r^{\prime }\right)  \label{E>A_3/2}
\end{equation}%
holds for $0\leq r\leq 1$. This conjecture was proved in 1997 in \cite[%
Theorem 2]{Qiu-JHIEE-3-1997} by Qiu and Shen (see also \cite[Theorem 1.1]%
{Barnard-JMA-31-2000}). Barnard et al. \cite{Barnard-JMA-32-2000} discovered
an upper bound $\left( \pi /2\right) A_{2}\left( 1,r^{\prime }\right) $ for $%
E\left( r\right) $, that is,%
\begin{equation}
E\left( r\right) \leq \frac{\pi }{2}A_{2}\left( 1,r^{\prime }\right) \text{
\ }\left( 0\leq r\leq 1\right) .  \label{E<A_2}
\end{equation}%
An improvement of (\ref{E<A_2}) was presented by Qiu in \cite[Corollary (1)]%
{Qiu-JHIEE-20(1)-2000} (see also \cite[Theorem 22.]{Alzer-JCAM-17(2)-2004}),
which states that the inequality%
\begin{equation}
E\left( r\right) \leq \frac{\pi }{2}A_{q_{0}}\left( 1,r^{\prime }\right)
\label{E<A_q0}
\end{equation}%
is valid for all $r\in \left( 0,1\right) $ with the best constant $q_{0}=\ln
2/\ln \left( \pi /2\right) $.

Motivated by the inequalities (\ref{E>A_3/2})--(\ref{E<A_q0}), some new
approximations for $E\left( r\right) $ in terms of bivariate means were
presented. For example, Chu and Wang \cite[Corollary 3.2]{Chu-RM-61-2012}
proved the inequality%
\begin{equation}
E\left( r\right) <\frac{\pi }{2}\mathcal{L}_{1/4}\left( 1,r^{\prime }\right)
=:\frac{\pi }{2}\mathcal{A}_{6}\left( r^{\prime }\right) ,  \label{E<L_1/4}
\end{equation}%
holds for $r\in \left( 0,1\right) $ with the best constant $1/4$, where $%
\mathcal{L}_{p}\left( a,b\right) $ is the Lehmer mean of positive $a$ and $b$
defined by%
\begin{equation*}
\mathcal{L}_{p}\left( a,b\right) =\frac{a^{p+1}+b^{p+1}}{a^{p}+b^{p}}.
\end{equation*}%
Soon afterwards, Chu et al. \cite[Theorem 3.1]{Chu-CMA-63-2012} gave a much
better upper bound for $E\left( r\right) $. They proved the double inequality%
\begin{eqnarray}
&&\tfrac{16-3\sqrt{2}\pi }{2\left( 2-\sqrt{2}\right) \pi }A\left(
1,r^{\prime }\right) -\tfrac{8-\pi -\sqrt{2}\pi }{2\left( 2-\sqrt{2}\right)
\pi }G\left( 1,r^{\prime }\right) +\tfrac{3\pi -8}{2\left( 2-\sqrt{2}\right)
\pi }S\left( 1,r^{\prime }\right)  \notag \\
&<&\frac{2}{\pi }E\left( r\right) <\frac{9}{8}A\left( 1,r^{\prime }\right) -%
\frac{5}{16}G\left( 1,r^{\prime }\right) +\frac{3}{16}S\left( 1,r^{\prime
}\right) =:\mathcal{A}_{7}\left( r^{\prime }\right)  \label{Ch-2}
\end{eqnarray}%
holds for $r\in \left( 0,1\right) $.

Wang et al. \cite[Theorem 2.4]{Wang-JAT-164-2012} established a sharp double
inequality for $E\left( r\right) $:%
\begin{equation}
\mathcal{A}_{2}\left( r^{\prime }\right) :=\frac{23}{16}A\left( 1,r^{\prime
}\right) -\frac{5}{16}H\left( 1,r^{\prime }\right) -\frac{1}{8}S\left(
1,r^{\prime }\right) <\frac{2}{\pi }E\left( r\right)  \label{Wang-1}
\end{equation}%
\begin{equation*}
<\tfrac{24-5\sqrt{2}\pi }{2\pi \left( 3-2\sqrt{2}\right) }A\left(
1,r^{\prime }\right) -\tfrac{8-\pi -\sqrt{2}\pi }{2\pi \left( 3-2\sqrt{2}%
\right) }H\left( 1,r^{\prime }\right) -\tfrac{\left( 16-5\pi \right) }{2\pi
\left( 3-2\sqrt{2}\right) }S\left( 1,r^{\prime }\right) ,
\end{equation*}%
and pointed out that the lower and upper bounds in (\ref{Wang-1}) are
stronger than ones in (\ref{E>A_3/2}) and (\ref{E<A_2}), respectively. In
2013, Wang and Chu \cite[Corollary 3.1]{Wang-JMAA-402-2013} presented
another improvement of (\ref{E>A_3/2}) and (\ref{E<A_2}), which states that%
\begin{equation}
\mathcal{A}_{3}\left( r^{\prime }\right) :=\frac{1}{128}\frac{\left(
9r^{\prime }{}^{2}+14r^{\prime }+9\right) ^{2}}{\left( r^{\prime }+1\right)
^{3}}<\frac{2}{\pi }E\left( r\right) <\frac{1}{\pi }\sqrt{4r^{\prime
}{}^{2}+\left( \pi ^{2}-8\right) r^{\prime }+4}.  \label{Wang-2}
\end{equation}

Very recently, Hua and Qi showed in \cite[Theorem 1.3.]{Hua-F-28(4)-2014}
that the double inequality%
\begin{equation}
\mathcal{A}_{4}\left( r^{\prime }\right) :=\frac{1+r^{\prime }+r^{\prime 2}}{%
2\left( 1+r^{\prime }\right) }+\frac{1+r^{\prime }}{8}<\frac{2}{\pi }E\left(
r\right) <\left( \frac{8}{\pi }-2\right) \frac{1+r^{\prime }+r^{\prime 2}}{%
1+r^{\prime }}+\left( 2-\frac{6}{\pi }\right) \left( 1+r^{\prime }\right)
\label{HQ}
\end{equation}%
is valid for $r\in \left( 0,1\right) $.

Other various approximations for $E\left( r\right) $ can be found in \cite%
{Barnard-JMAA-260-2001}, \cite{Kazi-JAT-146-2007}, \cite%
{Alzer-JCAM-17(2)-2004}, \cite{Barnard-JMA-32-2000}, \cite{Guo-MIA-14-2011}, 
\cite{Wang-AML-24-2011}, \cite{Chu-PIASMS-121-2011}, \cite{Chu-JMAA0395-2012}%
, \cite{Chu-AAA-11-2012}, \cite{Chu-JMI-7(2)-2013}, \cite{Song-JMI-7(4)-2013}%
, \cite{Li-JFSA-2013-394194}, \cite{Hua-PIAS(MS)-124(4)-2014}, \cite%
{Yin-AME-14-2014} and references therein.

For $a,b>0$ with $a\neq b$, the Stolarsky means $S_{p,q}(a,b)$ are defined
in \cite{Stolarsky-MM-48-1975} by%
\begin{equation}
S_{p,q}(a,b)=\left\{ 
\begin{array}{ll}
\left( {\dfrac{q(a^{p}-b^{p})}{p(a^{q}-b^{q})}}\right) ^{1/(p-q)} & \text{if 
}p\neq q,pq\neq 0, \\ 
\left( {\dfrac{a^{p}-b^{p}}{p(\ln a-\ln b)}}\right) ^{1/p} & \text{if }p\neq
0,q=0, \\ 
\left( {\dfrac{a^{q}-b^{q}}{q(\ln a-\ln b)}}\right) ^{1/q} & \text{if }%
p=0,q\neq 0, \\ 
\exp \left( \dfrac{a^{p}\ln a-b^{p}\ln b}{a^{p}-b^{p}}-\dfrac{1}{p}\right) & 
\text{if }{p=q\neq 0,} \\ 
\sqrt{ab} & \text{if }{p=q=0,}%
\end{array}%
\right.  \label{S_p,q}
\end{equation}%
and, $S_{p,q}(a,a)=a$. This family of means contains many famous means, for
example, $S_{1,0}(a,b)=L(a,b)$ --the logarithmic mean, $S_{1,1}(a,b)=I(a,b)$
--the identric (exponential) mean, $S_{2,1}(a,b)=A(a,b)$ --arithmetic mean, $%
S_{3/2,1/2}(a,b)=He(a,b)$ --Heronian mean, $%
S_{2p,p}(a,b)=A^{1/p}(a^{p},b^{p})=A_{p}$ --the $p$-order power mean, $%
S_{3p/2,p/2}(a,b)=He^{1/p}(a^{p},b^{p})=He_{p}$ --the $p$-order Heronian
mean, $S_{p,0}(a,b)=L^{1/p}(a^{p},b^{p})=L_{p}$ --the $p$-order logarithmic
mean, $S_{p,p}(a,b)=I^{1/p}(a^{p},b^{p})=I_{p}$ --the $p$-order identric
(exponential) mean, etc.

Stolarsky means have many well properties, which can follow directly from
the defining formula (\ref{S_p,q}) and be found in \cite%
{Stolarsky-MM-48-1975}, \cite{Leach-AMM-85-1978}, \cite{Pales-JMAA-131-1988d}%
, \cite{Qi-PAMS-130(6)-2002}, \cite{Yang-JIA-149286-2008}, \cite%
{Yang-IJMA-4(34)-2010}, \cite{Losonczi-PMD-75(1-2)-2009}, \cite%
{Yang-JISF-1(1)-2010}, \cite{Yang-IJMMS-2012-540710}, \cite%
{Yang-NKMS-49(1)-2012}, \cite{Yang-arxiv-1408.2252}. For later use, we
mentioned the following properties:

\begin{itemize}
\item[(P1)] For all $a,b>0$ and $p,q\in \mathbb{R}$, $S_{p,q}(a,b)$ are
increasing with both $p$ and $q$, or with both $a$ and $b$ (see \cite[%
(2.22), (3.12)]{Leach-AMM-85-1978}).

\item[(P2)] For fixed $c>0$, $S_{p,2c-p}\left( a,b\right) $ is increasing in 
$p$ on $(-\infty ,c]$ and decreasing on $[c,\infty )$ (see \cite[(3.14)]%
{Leach-AMM-85-1978}, \cite[Corollary 1.1]{Yang-NKMS-49(1)-2012}).

\item[(P3)] For fixed $c>0$ and $p\in \left( 0,2c\right) $, $\left( 1/\theta
_{p}\right) S_{p,2c-p}\left( a,b\right) $ is decreasing in $p$ on $(0,c)$
and increasing on $(c,2c)$, where $\theta _{p}$ is defined by%
\begin{equation}
\theta _{p}=\left( \frac{2c-p}{p}\right) ^{1/\left( 2p-2c\right) }\text{ if }%
p\neq c\text{ and }\theta _{c}=e^{-1/c}  \label{theta_p(c)}
\end{equation}%
(see \cite[Corollary 1.2]{Yang-NKMS-49(1)-2012}).
\end{itemize}

\begin{remark}
\label{P3-C}Taking $a=b$ in (P3) yields that $\left( 1/\theta _{p}\right)
S_{p,2c-p}\left( a,a\right) $ is decreasing in $p$ on $(0,c)$ and increasing
on $(c,2c)$. Since $S_{p,2c-p}\left( a,a\right) =a$, it follows that $\theta
_{p}$ strictly increasing in $p$ on $\left( 0,c\right) $ and decreasing on $%
\left( c,2c\right) $.
\end{remark}

Now we intend to estimate for the complete elliptic integrals of the second
kind $E\left( r\right) $ by the Stolarsky means of $1$ and $r^{\prime }$,
i.e. $S_{p,q}\left( 1,r^{\prime }\right) $. Expanding in power series gives%
\begin{eqnarray*}
\frac{2}{\pi }E\left( r\right) -S_{p,q}\left( 1,r^{\prime }\right) &=&-\frac{%
p+q-9/2}{96}r^{4}-\frac{p+q-9/2}{128}r^{6} \\
&&+\frac{8\left( p+q\right) \left( 2p^{2}+2q^{2}-5p-5q-550\right) +19\,845}{%
45\times 2^{14}}r^{8}+O\left( r^{10}\right) .
\end{eqnarray*}%
In order to increase accuracy of estimate for $E\left( r\right) $, we let $%
p+q-9/2=0$, or $q=9/2-p$. Then we get%
\begin{equation}
\frac{2}{\pi }E\left( r\right) -S_{9/2-p,p}\left( 1,r^{\prime }\right) =%
\frac{\left( 4p-7\right) \left( 4p-11\right) }{5\times 2^{14}}r^{8}+O\left(
r^{10}\right) .  \label{E-S_p}
\end{equation}

Further, taking $p=7/4$ or $11/4$ yields%
\begin{equation}
S_{11/4,7/4}\left( 1,r^{\prime }\right) =\frac{7}{11}\frac{1-r^{\prime 11/4}%
}{1-r^{\prime 7/4}}:=\mathcal{A}_{5}\left( r^{\prime }\right) ,
\label{S_11/4,7/4}
\end{equation}%
and%
\begin{eqnarray*}
\frac{2}{\pi }E\left( r\right) -S_{11/4,7/4}\left( 1,r^{\prime }\right) &=&%
\frac{1}{7\times 2^{21}}r^{12}+O\left( r^{14}\right) , \\
\lim_{r\rightarrow 1^{-}}\left( \frac{2}{\pi }E\left( r\right)
-S_{11/4,7/4}\left( 1,r^{\prime }\right) \right) &=&\frac{2}{\pi }-\frac{7}{%
11}\approx 0.00025614.
\end{eqnarray*}%
Letting $p=2$ or $5/2$ yields%
\begin{equation}
S_{5/2,2}\left( 1,r^{\prime }\right) =\left( \frac{4}{5}\frac{1-r^{\prime
5/2}}{1-r^{\prime 2}}\right) ^{2}:=\mathcal{A}_{8}\left( r^{\prime }\right)
\label{S_5/2,2}
\end{equation}%
and%
\begin{eqnarray*}
\frac{2}{\pi }E\left( r\right) -S_{5/2,2}\left( 1,r^{\prime }\right) &=&-%
\frac{3}{5\times 2^{14}}r^{8}+O\left( r^{10}\right) , \\
\lim_{r\rightarrow 1^{-}}\left( \frac{2}{\pi }E\left( r\right)
-S_{5/2,2}\left( 1,r^{\prime }\right) \right) &=&\frac{2}{\pi }-\frac{16}{25}%
\approx -0.0033802.
\end{eqnarray*}

These show that $S_{11/4,7/4}\left( 1,r^{\prime }\right) $ and $%
S_{5/2,2}\left( 1,r^{\prime }\right) $ may be excellent approximations for
the complete elliptic integrals of the second kind. The purpose of the paper
is to prove this assertion. Our main results are contained in the following
theorems.

\begin{theorem}
\label{MT-7/4}The function%
\begin{equation*}
F\left( r\right) =\frac{1-\left( 2/\pi \right) E\left( r\right) }{%
1-S_{11/4,7/4}\left( 1,r^{\prime }\right) }
\end{equation*}%
is strictly decreasing from $\left( 0,1\right) $ onto $\left( 11\left( \pi
-2\right) /\left( 4\pi \right) ,1\right) $. Therefore, the double inequality%
\begin{equation}
1-\mu +\mu S_{11/4,7/4}\left( 1,r^{\prime }\right) <\frac{2}{\pi }E\left(
r\right) <1-\lambda +\lambda S_{11/4,7/4}\left( 1,r^{\prime }\right)
\label{MI-7/4-1}
\end{equation}%
holds if and only if $\mu \geq 1$ and $\lambda \leq 11\left( \pi -2\right)
/\left( 4\pi \right) \approx 0.9993$.

In particular, we have%
\begin{equation}
S_{11/4,7/4}\left( 1,r^{\prime }\right) <\tfrac{2}{\pi }E\left( r\right) <%
\tfrac{22-7\pi }{4\pi }+\tfrac{11\left( \pi -2\right) }{4\pi }%
S_{11/4,7/4}\left( 1,r^{\prime }\right) <\tfrac{22}{7\pi }S_{11/4,7/4}\left(
1,r^{\prime }\right) ,  \label{MI-7/4-2}
\end{equation}%
where the coefficients $1$ and $22/\left( 7\pi \right) \approx 1.\,0004$ are
the best constants.
\end{theorem}

\begin{theorem}
\label{MT-2}The function%
\begin{equation*}
G\left( r\right) =\frac{1-\left( 2/\pi \right) E\left( r\right) }{%
1-S_{5/2,2}\left( 1,r^{\prime }\right) }
\end{equation*}%
is strictly increasing from $\left( 0,1\right) $ onto $\left( 1,25\left( \pi
-2\right) /\left( 9\pi \right) \right) $. Consequently, the double inequality%
\begin{equation}
1-\xi +\xi S_{5/2,2}\left( 1,r^{\prime }\right) <\frac{2}{\pi }E\left(
r\right) <1-\eta +\eta S_{5/2,2}\left( 1,r^{\prime }\right)  \label{MI-2-1}
\end{equation}%
holds if and only if $\xi \geq 25\left( \pi -2\right) /\left( 9\pi \right)
\approx 1.\,0094$ and $\eta \leq 1$.

Particularly, it holds that%
\begin{equation}
\tfrac{25}{8\pi }S_{5/2,2}\left( 1,r^{\prime }\right) <-\tfrac{2\left( 8\pi
-25\right) }{9\pi }+\tfrac{25\left( \pi -2\right) }{9\pi }S_{5/2,2}\left(
1,r^{\prime }\right) <\tfrac{2}{\pi }E\left( r\right) <S_{5/2,2}\left(
1,r^{\prime }\right) ,  \label{MI-2-2}
\end{equation}%
where the coefficients $25/\left( 8\pi \right) \approx 0.99472$ and $1$ are
the best constants.
\end{theorem}

The paper is organized as follows. Some lemmas used to prove main results
are presented in Section 2. The proof of Theorem 1 is complicated and
longer, so it is independently arranged in Sections 3; while the proof of
Theorem 2 is placed in Section 4. In Section 5, some interesting and applied
corollaries involving the monotonicity of difference and ratio between $%
\left( 2/\pi \right) E\left( r\right) $ and $S_{9/2-p,p}\left( 1,r^{\prime
}\right) $ are deduced. In the last section, it is shown that our
approximations $S_{11/4,7/4}\left( 1,r^{\prime }\right) $ and $%
S_{2,5/2}\left( 1,r^{\prime }\right) $ are indeed excellent by accuracy of
comparing some known approximations with ours.

\section{Lemmas}

In order to prove our main results, we need some lemmas.

The first lemma is the "L'Hospital Monotone Rule" (see \cite%
{Pinelis-JIPAM-3(1)-2002-5}, \cite{Anderson-AMM-113(9)-2006}), which has
been widely used very effectively in the study of some areas. A new and
natural way to prove this class of rules can refer to \cite%
{Yang-arxiv-1409.6408}, which is easily understood and used.

\begin{lemma}[{\protect\cite[Proposition 1.1]{Pinelis-JIPAM-3(1)-2002-5}, 
\protect\cite[Theorem 2]{Anderson-AMM-113(9)-2006}}]
\label{L-f/g}For $-\infty <a<b<\infty $, let $f,g:[a,b]\rightarrow \mathbb{R}
$ be continuous functions that are differentiable on $\left( a,b\right) $,
with $f\left( a\right) =g\left( a\right) =0$ or $f\left( b\right) =g\left(
b\right) =0$. Assume that $g^{\prime }(x)\neq 0$ for each $x$ in $(a,b)$. If 
$f^{\prime }/g^{\prime }$ is increasing (decreasing) on $(a,b)$ then so is $%
f/g$.
\end{lemma}

The second lemma is a monotonicity criterion for the ratio of power series,
which will be used to prove Theorem 2.

\begin{lemma}[\textbf{\protect\cite{Biernacki-AUMC-S-9-1955}}]
\label{L-ps-A/B}Let $A\left( t\right) =\sum_{k=0}^{\infty }a_{k}t^{k}$ and $%
B\left( t\right) =\sum_{k=0}^{\infty }b_{k}t^{k}$ be two real power series
converging on $\left( -r,r\right) $ ($r>0$) with $b_{k}>0$ for all $k$. If
the sequence $\{a_{k}/b_{k}\}\ $is increasing (decreasing) for all $k$, then
the function $t\mapsto A\left( t\right) /B\left( t\right) $ is also
increasing (decreasing) on $\left( 0,r\right) $.
\end{lemma}

A more general monotonicity criterion for the ratio of power series has been
established in \cite{Yanga-JMAA-428-2014} recently.

For $x>0$ the classical Euler's gamma function $\Gamma $ and psi (digamma)
function $\psi $ are defined by%
\begin{equation}
\Gamma \left( x\right) =\int_{0}^{\infty }t^{x-1}e^{-t}dt,\text{ \ \ \ \ }%
\psi \left( x\right) =\frac{\Gamma ^{\prime }\left( x\right) }{\Gamma \left(
x\right) },  \label{Gamma}
\end{equation}%
respectively. In the proof of Theorem 1, we will use several inequalities
for the gamma and psi functions, in which Lemma \ref{L-g(x+a)/g(x+1)>} is
very crucial.

\begin{lemma}[{\protect\cite{Wendel-AMM-55-1948}, \protect\cite[(2.8)]%
{Qi-JIA-2010-493058}}]
\label{L-g(x+a)/g(x+1)>}For all $x>0$ and all $a\in \left( 0,1\right) $, it
holds that%
\begin{equation*}
\left( \frac{x}{x+a}\right) ^{1-a}<\frac{\Gamma \left( x+a\right) }{%
x^{a}\Gamma \left( x\right) }<1,
\end{equation*}%
or equivalently,%
\begin{equation}
\frac{1}{\left( x+a\right) ^{1-a}}<\frac{\Gamma \left( x+a\right) }{\Gamma
\left( x+1\right) }<\frac{1}{x^{1-a}}.  \label{L-gi><}
\end{equation}
\end{lemma}

\begin{lemma}[\protect\cite{Ivady-JMI-3(2)-2009}]
\label{L-g(x)<}For $x\in \left( 0,1\right) $, it holds that%
\begin{equation*}
\frac{x^{2}+1}{x+1}<\Gamma \left( x+1\right) <\frac{x^{2}+2}{x+2},x\in
\left( 0,1\right) ,
\end{equation*}%
or equivalently,%
\begin{equation}
\frac{x^{2}+1}{\left( x+1\right) x}<\Gamma \left( x\right) <\frac{x^{2}+2}{%
\left( x+2\right) x},x\in \left( 0,1\right) .  \label{L-gi<>}
\end{equation}
\end{lemma}

\begin{lemma}[{\protect\cite[Lemma 1.7]{Batir-AM-91-2008}}]
\label{L-p(x+1)>}Let $x$ be a positive real number. Then we have 
\begin{equation*}
\psi (x+1)>\ln (x+\frac{1}{2}).
\end{equation*}
\end{lemma}

\begin{lemma}[{\protect\cite[Lemma 7]{Yang-AAA-702718-AAA}, \protect\cite[%
Lemma 1]{Yang-arxiv-1409.6408}}]
\label{L-P-zp}Let $n\in \mathbb{N}$ and $m\in \mathbb{N}\cup \{0\}$ with $%
n>m $ and let $P_{n}\left( t\right) $ be an $n$ degrees polynomial defined
by 
\begin{equation}
P_{n}\left( t\right) =\sum_{i=m+1}^{n}a_{i}t^{i}-\sum_{i=0}^{m}a_{i}t^{i},
\label{2.4}
\end{equation}%
where $a_{n},a_{m}>0$, $a_{i}\geq 0$ for $0\leq i\leq n-1$ with $i\neq m$.
Then there is a unique number $t_{m+1}\in \left( 0,\infty \right) $ to
satisfy $P_{n}\left( t\right) =0$ such that $P_{n}\left( t\right) <0$ for $%
t\in \left( 0,t_{m+1}\right) $ and $P_{n}\left( t\right) >0$ for $t\in
\left( t_{m+1},\infty \right) $.
\end{lemma}

\section{Proof of Theorem \protect\ref{MT-7/4}}

We are in a position to prove Theorem \ref{MT-7/4}.

\begin{proof}[Proof of Theorem \protect\ref{MT-7/4}]
Let us consider the ratio%
\begin{equation*}
F\left( r\right) =\frac{1-\left( 2/\pi \right) E\left( r\right) }{%
1-S_{11/4,7/4}\left( 1,r^{\prime }\right) }=\frac{11\left( 1-r^{\prime
7/4}\right) \left( 1-\left( 2/\pi \right) E\right) }{11\left( 1-r^{\prime
7/4}\right) -7\left( 1-r^{\prime 11/4}\right) }:=\frac{f_{1}\left( r\right) 
}{f_{2}\left( r\right) }.
\end{equation*}%
Clearly, $f_{1}\left( 0^{+}\right) =f_{2}\left( 0^{+}\right) =0$.

Differentiations by the formulas%
\begin{equation}
\frac{dK}{dr}=\frac{E-r^{\prime ^{2}}K}{rr^{\prime ^{2}}}\text{, \ }\frac{dE%
}{dr}=\frac{E-K}{r}\text{, \ }\frac{d\left( K-E\right) }{dr}=\frac{rE}{%
r^{\prime ^{2}}}  \label{dK-dE}
\end{equation}%
yield%
\begin{eqnarray*}
\frac{f_{1}^{\prime }\left( r\right) }{f_{2}^{\prime }\left( r\right) } &=&%
\frac{1-\left( 2/\pi \right) E+8\left( r^{\prime 1/4}-r^{\prime 2}\right)
\left( K-E\right) /\left( 7\pi r^{2}\right) }{1-r^{\prime }}:=\frac{%
f_{3}\left( r\right) }{f_{4}\left( r\right) }, \\
f_{3}\left( 0^{+}\right) &=&f_{4}\left( 0^{+}\right) =0,
\end{eqnarray*}%
\begin{eqnarray*}
\frac{7\pi }{2}\frac{f_{3}^{\prime }\left( r\right) }{f_{4}^{\prime }\left(
r\right) } &=&\frac{r^{\prime }}{r}\left( 4\frac{r^{\prime 1/4}-r^{\prime 2}%
}{rr^{\prime 2}}E-\frac{\left( 8-7r^{2}\right) r^{\prime 1/4}-8r^{\prime 2}}{%
r^{3}r^{\prime 2}}\left( K-E\right) -7\frac{E-K}{r}\right) \\
&=&\frac{\left( 8-3r^{2}\right) E-\left( 8-7r^{2}\right) K}{r^{4}r^{\prime
3/4}}+\left( \left( 7r^{2}+8\right) K-\left( 11r^{2}+8\right) E\right) \frac{%
r^{\prime }}{r^{4}},
\end{eqnarray*}%
\begin{eqnarray*}
\frac{7\pi }{2}\left( \frac{f_{3}^{\prime }\left( r\right) }{f_{4}^{\prime
}\left( r\right) }\right) ^{\prime } &=&-\frac{%
32K-32E-14r^{2}K-3r^{4}K-2r^{2}E}{r^{5}r^{\prime }} \\
&&+\frac{1}{4}\frac{128K-128E-224r^{2}K+93r^{4}K+160r^{2}E-21r^{4}E}{%
r^{5}r^{\prime 11/4}} \\
&:&=-\frac{f_{5}\left( r\right) -r^{\prime -7/4}f_{6}\left( r\right) }{%
r^{5}r^{\prime }}=:-\frac{f_{7}\left( r\right) }{r^{5}r^{\prime }},
\end{eqnarray*}%
where%
\begin{eqnarray*}
f_{5}\left( r\right) &=&32K-32E-14r^{2}K-3r^{4}K-2r^{2}E, \\
f_{6}\left( r\right) &=&\frac{1}{4}\left(
128K-128E-224r^{2}K+93r^{4}K+160r^{2}E-21r^{4}E\right) .
\end{eqnarray*}

If we can prove that $f_{7}\left( r\right) >0$ for $r\in \left( 0,1\right) $%
, that is, the function $f_{3}^{\prime }/f_{4}^{\prime }$ is decreasing on $%
\left( 0,1\right) $, then so is $f_{1}^{\prime }/f_{2}^{\prime }$ by Lemma %
\ref{L-f/g}, which in turn implies that $f_{1}/f_{2}$ by using Lemma \ref%
{L-f/g} again, that is, the function $F$ is strictly decreasing on $\left(
0,1\right) $. Next we prove that $f_{7}\left( r\right) >0$ for $r\in \left(
0,1\right) $ stepwise.

\textbf{Step 1}: Expanding $f_{7}\left( r\right) $ in power series. By the
expansions (\ref{Ke}) and (\ref{Ee}), we have%
\begin{eqnarray*}
f_{5}\left( r\right) &=&32K-32E-14Kr^{2}-3Kr^{4}-2r^{2}E \\
&=&\frac{\pi }{2}\left( 32\sum_{n=0}^{\infty }\frac{\left( \frac{1}{2}%
\right) _{n}^{2}}{n!n!}r^{2n}-32\sum_{n=0}^{\infty }\frac{\left( -\frac{1}{2}%
\right) _{n}\left( \frac{1}{2}\right) _{n}}{n!n!}r^{2n}\right. \\
&&\left. -14\sum_{n=1}^{\infty }\tfrac{\left( \frac{1}{2}\right) _{n-1}^{2}}{%
\left( n-1\right) !\left( n-1\right) !}r^{2n}-3\sum_{n=2}^{\infty }\tfrac{%
\left( \frac{1}{2}\right) _{n-2}^{2}}{\left( n-2\right) !\left( n-2\right) !}%
r^{2n}-2\sum_{n=1}^{\infty }\tfrac{\left( -\frac{1}{2}\right) _{n-1}\left( 
\frac{1}{2}\right) _{n-1}}{\left( n-1\right) !\left( n-1\right) !}%
r^{2n}\right) \\
&=&\frac{3\pi }{2}\sum_{n=3}^{\infty }\frac{n\left( 5n-6\right) \left(
n-1\right) \left( n-2\right) \left( \frac{1}{2}\right) _{n-2}^{2}}{n!n!}%
r^{2n}:=\frac{3\pi }{2}r^{6}\sum_{n=0}^{\infty }a_{n}r^{2n},
\end{eqnarray*}%
where%
\begin{equation}
a_{n}=\frac{\left( 5n+9\right) \left( \frac{1}{2}\right) _{n+1}^{2}}{%
n!\left( n+3\right) !}=\frac{5n+9}{n!\left( n+3\right) !}\left( \frac{\Gamma
\left( n+3/2\right) }{\Gamma \left( 1/2\right) }\right) ^{2};  \label{an}
\end{equation}%
\begin{eqnarray*}
f_{6}\left( r\right) &=&\frac{1}{4}\left(
128K-128E-224r^{2}K+93r^{4}K+160r^{2}E-21r^{4}E\right) \\
&=&\frac{\pi }{8}\left( 128\sum_{n=0}^{\infty }\frac{\left( \frac{1}{2}%
\right) _{n}^{2}}{n!n!}r^{2n}-128\sum_{n=0}^{\infty }\frac{\left( -\frac{1}{2%
}\right) _{n}\left( \frac{1}{2}\right) _{n}}{n!n!}r^{2n}\right. \\
&&-224\sum_{n=1}^{\infty }\frac{\left( \frac{1}{2}\right) _{n-1}^{2}}{\left(
n-1\right) !\left( n-1\right) !}r^{2n}+93\sum_{n=2}^{\infty }\frac{\left( 
\frac{1}{2}\right) _{n-2}^{2}}{\left( n-2\right) !\left( n-2\right) !}r^{2n}
\\
&&\left. +160\sum_{n=1}^{\infty }\frac{\left( -\frac{1}{2}\right)
_{n-1}\left( \frac{1}{2}\right) _{n-1}}{\left( n-1\right) !\left( n-1\right)
!}r^{2n}-21\sum_{n=2}^{\infty }\frac{\left( -\frac{1}{2}\right) _{n-2}\left( 
\frac{1}{2}\right) _{n-2}}{\left( n-2\right) !\left( n-2\right) !}%
r^{2n}\right) \\
&=&-\frac{3\pi }{16}\sum_{n=3}^{\infty }\frac{n\left( n-1\right) \left(
n-2\right) \left( n-4\right) \left( n+15\right) \left( 2n-5\right) \left( 
\frac{1}{2}\right) _{n-3}^{2}}{n!n!}r^{2n} \\
&:&=-\frac{3\pi }{16}r^{6}\sum_{n=0}^{\infty }c_{n}r^{2n},
\end{eqnarray*}%
where%
\begin{equation}
c_{n}=\frac{\left( n-1\right) \left( n+18\right) \left( 2n+1\right) \left( 
\frac{1}{2}\right) _{n}^{2}}{n!\left( n+3\right) !}=\frac{\left( n-1\right)
\left( n+18\right) \left( 2n+1\right) }{n!\left( n+3\right) !}\left( \frac{%
\Gamma \left( n+1/2\right) }{\Gamma \left( 1/2\right) }\right) ^{2}.
\label{cn}
\end{equation}%
Also, using binomial series we have%
\begin{equation}
r^{\prime -7/4}=\left( 1-r^{2}\right) ^{-7/8}=\sum_{n=0}^{\infty }\frac{%
\left( \frac{7}{8}\right) _{n}}{n!}r^{2n}:=\sum_{n=0}^{\infty }b_{n}r^{2n}.
\label{bn}
\end{equation}

Then making use of the Cauchy product gives%
\begin{eqnarray*}
f_{7}\left( r\right) &=&f_{5}\left( r\right) -\left( 1-r^{2}\right)
^{-7/8}f_{6}\left( r\right) =\frac{3\pi }{2}r^{6}\sum_{n=0}^{\infty
}a_{n}r^{2n}+\frac{3\pi }{16}r^{6}\sum_{n=0}^{\infty
}c_{n}r^{2n}\sum_{n=0}^{\infty }b_{n}r^{2n} \\
&=&\frac{3\pi }{16}r^{6}\sum_{n=0}^{\infty }\left(
8a_{n}+\sum_{k=0}^{n}b_{n-k}c_{k}\right) r^{2n}:=\frac{3\pi }{16}%
r^{6}\sum_{n=0}^{\infty }d_{n}r^{2n}.
\end{eqnarray*}

\textbf{Step 2}: A simple verification yields $d_{0}=d_{1}=d_{2}=d_{3}=0$
and $d_{4}$, $d_{5}$, $d_{6}$, $d_{7}$, $d_{8}$, $d_{9}$, $d_{10}$ are equal
to%
\begin{equation*}
\tfrac{35}{32\,768},\tfrac{903}{262\,144},\tfrac{7343}{1048\,576},\tfrac{%
193\,225}{16\,777\,216},\tfrac{36\,001\,035}{2147\,483\,648},\tfrac{%
387\,471\,275}{17\,179\,869\,184},\tfrac{7897\,834\,945}{274\,877\,906\,944},
\end{equation*}%
respectively.

\textbf{Step 3}: We prove that for $n\geq 10$,%
\begin{equation}
D_{n}:=\frac{8}{7}d_{n+1}-d_{n}>g\left( n\right) ,  \label{en>gn}
\end{equation}%
where%
\begin{equation}
g\left( n\right) =\frac{4}{\pi }\frac{n\left( 2n+1\right) }{\left(
n+1\right) \left( n+2\right) \left( n+3\right) }-\frac{3}{7\Gamma \left(
7/8\right) }\frac{n^{7/8}}{n+1}+\frac{128}{7\pi }\frac{n^{7/8}}{64n-9}%
g_{1}\left( n\right) ,  \label{gn}
\end{equation}%
here%
\begin{equation}
g_{1}\left( n\right) =\sum_{k=2}^{n}\frac{n-k}{n-k+1}\frac{\left( k-1\right)
\left( k+18\right) }{\left( k+1\right) \left( k+2\right) \left( k+3\right) }%
:=\sum_{k=2}^{n}\alpha _{n-k}\beta _{k}.  \label{g1n}
\end{equation}

To this end, we note that for $k\geq 0$,%
\begin{eqnarray*}
a_{k+1} &=&\frac{1}{4}\frac{\left( 2k+3\right) ^{2}\left( 5k+14\right) }{%
\left( 5k+9\right) \left( k+1\right) \left( k+4\right) }a_{k}, \\
b_{k+1} &=&\frac{k+7/8}{k+1}b_{k}, \\
c_{k+1} &=&\frac{1}{4}\frac{k\left( 2k+1\right) \left( 2k+3\right) \left(
k+19\right) }{\left( k-1\right) \left( k+1\right) \left( k+4\right) \left(
k+18\right) }c_{k}\text{ (}k\geq 2\text{)}.
\end{eqnarray*}%
Also, it is seen that $c_{0}=-3$, $c_{1}=0$, $c_{k}>0$ for $k\geq 2$, and $%
a_{k},b_{k}>0$ for $k\geq 0$. Then the sequence $D_{n}$ can be expressed as%
\begin{eqnarray*}
D_{n} &=&\frac{8}{7}d_{n+1}-d_{n}=\frac{8}{7}\left(
8a_{n+1}+b_{n+1}c_{0}+\sum_{k=2}^{n+1}b_{n+1-k}c_{k}\right) -\left(
8a_{n}+b_{n}c_{0}+\sum_{k=2}^{n}b_{n-k}c_{k}\right) \\
&=&8\left( \frac{8}{7}a_{n+1}-a_{n}\right) -3\left( \frac{8}{7}%
b_{n+1}-b_{n}\right) +\frac{8}{7}b_{0}c_{n+1}+\sum_{k=2}^{n}\left( \frac{8}{7%
}b_{n+1-k}-b_{n-k}\right) c_{k}
\end{eqnarray*}%
\begin{eqnarray*}
&=&8\left( \frac{8}{7}\frac{1}{4}\frac{\left( 2n+3\right) ^{2}\left(
5n+14\right) }{\left( 5n+9\right) \left( n+1\right) \left( n+4\right) }%
-1\right) a_{n}-3\left( \frac{8}{7}\frac{n+7/8}{n+1}-1\right) b_{n} \\
&&+\frac{8}{7}\frac{1}{4}\frac{n\left( 2n+1\right) \left( 2n+3\right) \left(
n+19\right) }{\left( n-1\right) \left( n+1\right) \left( n+4\right) \left(
n+18\right) }c_{n}+\sum_{k=2}^{n}\left( \frac{8}{7}\frac{n-k+7/8}{n-k+1}%
-1\right) b_{n-k}c_{k}
\end{eqnarray*}%
\begin{eqnarray*}
&=&\frac{8}{7}\frac{n\left( 5n^{2}-6n-29\right) }{\left( 5n+9\right) \left(
n+1\right) \left( n+4\right) }a_{n}-\frac{3}{7}\frac{n}{n+1}b_{n} \\
&&+\frac{2}{7}\frac{n\left( 2n+1\right) \left( 2n+3\right) \left(
n+19\right) }{\left( n-1\right) \left( n+1\right) \left( n+4\right) \left(
n+18\right) }c_{n}+\frac{1}{7}\sum_{k=2}^{n}\frac{n-k}{n-k+1}b_{n-k}c_{k}.
\end{eqnarray*}

Notice that coefficient of $a_{n}$ is positive due to $5n^{2}-6n-29>0$ for $%
n\geq 4$, by the double inequality (\ref{L-gi><}) in Lemma \ref%
{L-g(x+a)/g(x+1)>} it is derived that for $k\geq 0$,%
\begin{eqnarray*}
a_{k} &>&\frac{1}{2\pi }\frac{\left( 2k+1\right) \left( 5k+9\right) }{\left(
k+1\right) \left( k+2\right) \left( k+3\right) }, \\
\frac{1}{\left( k+7/8\right) ^{1/8}} &<&b_{k}<\frac{1}{k^{1/8}\Gamma \left(
7/8\right) }, \\
c_{k} &>&\frac{2}{\pi }\frac{\left( k-1\right) \left( k+18\right) }{\left(
k+1\right) \left( k+2\right) \left( k+3\right) }\text{ (}k\geq 2\text{).}
\end{eqnarray*}%
Applying these inequalities to the expression of $D_{n}$ yields%
\begin{eqnarray}
D_{n} &>&\frac{8}{7}\frac{n\left( 5n^{2}-6n-29\right) }{\left( 5n+9\right)
\left( n+1\right) \left( n+4\right) }\frac{1}{2\pi }\frac{\left( 2n+1\right)
\left( 5n+9\right) }{\left( n+1\right) \left( n+2\right) \left( n+3\right) }-%
\frac{3}{7}\frac{n}{n+1}\frac{1}{n^{1/8}\Gamma \left( 7/8\right) }  \notag \\
&&+\frac{2}{7}\frac{n\left( 2n+1\right) \left( 2n+3\right) \left(
n+19\right) }{\left( n-1\right) \left( n+1\right) \left( n+4\right) \left(
n+18\right) }\frac{2}{\pi }\frac{\left( n-1\right) \left( n+18\right) }{%
\left( n+1\right) \left( n+2\right) \left( n+3\right) }  \notag \\
&&+\frac{1}{7}\sum_{k=2}^{n}\left( \frac{n-k}{n-k+1}\frac{1}{\left(
n-k+7/8\right) ^{1/8}}\frac{2}{\pi }\frac{\left( k-1\right) \left(
k+18\right) }{\left( k+1\right) \left( k+2\right) \left( k+3\right) }\right)
.  \label{en>}
\end{eqnarray}

By the known inequality%
\begin{equation*}
\left( 1-x\right) ^{p}<1-px\text{, \ }p,x\in \left( 0,1\right)
\end{equation*}%
it is deduced that for $2\leq k\leq n$,%
\begin{eqnarray*}
\frac{1}{\left( n-k+7/8\right) ^{1/8}} &\geq &\frac{1}{\left( n-2+7/8\right)
^{1/8}}=\frac{1}{n^{1/8}\left( 1-9/\left( 8n\right) \right) ^{1/8}} \\
&>&\frac{1}{n^{1/8}\left( 1-9/\left( 64n\right) \right) }=64\frac{n^{7/8}}{%
64n-9},
\end{eqnarray*}%
which is used to the last member of the right hand side in (\ref{en>}) and
factoring gives%
\begin{eqnarray*}
D_{n} &>&\frac{4}{\pi }\frac{n\left( 2n+1\right) }{\left( n+1\right) \left(
n+2\right) \left( n+3\right) }-\frac{3}{7\Gamma \left( 7/8\right) }\frac{%
n^{7/8}}{n+1} \\
&&+\frac{128}{7\pi }\frac{n^{7/8}}{64n-9}\sum_{k=2}^{n}\left( \frac{n-k}{%
n-k+1}\frac{\left( k-1\right) \left( k+18\right) }{\left( k+1\right) \left(
k+2\right) \left( k+3\right) }\right) ,
\end{eqnarray*}%
which proves the step.

\textbf{Step 4}: The sequence $g_{1}\left( n\right) :=\sum_{k=2}^{n}\alpha
_{n-k}\beta _{k}$ defined by (\ref{g1n}) can be expressed as%
\begin{equation}
g_{1}\left( n\right) =\tfrac{n^{3}+7n^{2}-12n+24}{\left( n+2\right) \left(
n+3\right) \left( n+4\right) }\left( \psi \left( n+1\right) +\gamma \right) -%
\tfrac{n\left( 11n^{2}+8n+21\right) }{\left( n+1\right) \left( n+2\right)
\left( n+3\right) \left( n+4\right) },  \label{g1n-e}
\end{equation}%
where $\psi \left( t\right) $ denotes the psi function, $\gamma $ is the
Euler's constant.

In fact, decomposing rational function $\alpha _{n-k}\beta _{k}$ into
partial fractions gives%
\begin{eqnarray*}
\alpha _{n-k}\beta _{k} &=&\frac{n-k}{n-k+1}\frac{\left( k-1\right) \left(
k+18\right) }{\left( k+1\right) \left( k+2\right) \left( k+3\right) } \\
&=&-17\frac{n+1}{n+2}\times \frac{1}{k+1}+48\frac{n+2}{n+3}\times \frac{1}{%
k+2} \\
&&-30\frac{n+3}{n+4}\times \frac{1}{k+3}-\frac{n\left( n+19\right) }{\left(
n+2\right) \left( n+3\right) \left( n+4\right) }\times \frac{1}{n-k+1}.
\end{eqnarray*}%
Hence, we get%
\begin{eqnarray*}
\sum_{k=2}^{n}\alpha _{n-k}\beta _{k} &=&-17\frac{n+1}{n+2}\left(
\sum_{k=1}^{n}\frac{1}{k}+\frac{1}{n+1}-\frac{3}{2}\right) \\
&&+48\frac{n+2}{n+3}\left( \sum_{k=1}^{n}\frac{1}{k}+\frac{1}{n+1}+\frac{1}{%
n+2}-\frac{11}{6}\right) \\
&&-30\frac{n+3}{n+4}\left( \sum_{k=1}^{n}\frac{1}{k}+\frac{1}{n+1}+\frac{1}{%
n+2}+\frac{1}{n+3}-\frac{25}{12}\right) \\
&&-\frac{n\left( n+19\right) }{\left( n+2\right) \left( n+3\right) \left(
n+4\right) }\left( \sum_{k=1}^{n}\frac{1}{k}-\frac{1}{n}\right) \\
&=&\frac{n^{3}+7n^{2}-12n+24}{\left( n+2\right) \left( n+3\right) \left(
n+4\right) }\sum_{k=1}^{n}\frac{1}{k}-\frac{n\left( 11n^{2}+8n+21\right) }{%
\left( n+1\right) \left( n+2\right) \left( n+3\right) \left( n+4\right) },
\end{eqnarray*}%
which by the identity $\sum_{k=1}^{n}\frac{1}{k}=\psi \left( n+1\right)
+\gamma $ proves the step.

\textbf{Step 5}: We show that the sequence $g\left( n\right) $ defined by (%
\ref{gn}) satisfies the inequality%
\begin{equation}
g\left( n\right) >\frac{128}{7\pi }\frac{n^{7/8}}{64n-9}\frac{%
n^{3}+7n^{2}-12n+24}{\left( n+2\right) \left( n+3\right) \left( n+4\right) }%
g_{2}\left( n\right) ,  \label{gn>}
\end{equation}%
for $n\geq 4$, where%
\begin{eqnarray*}
g_{2}\left( n\right) &=&\ln \left( n+1/2\right) +\gamma -\frac{n\left(
11n^{2}+8n+21\right) }{\left( n+1\right) \left( n^{3}+7n^{2}-12n+24\right) }
\\
&&+\frac{7}{32}\sqrt[8]{n}\tfrac{\left( n+4\right) \left( 2n+1\right) \left(
64n-9\right) }{\left( n+1\right) \left( n^{3}+7n^{2}-12n+24\right) }-\frac{%
3\pi }{128\Gamma \left( 7/8\right) }\tfrac{\left( 64n-9\right) \left(
n+2\right) \left( n+3\right) \left( n+4\right) }{\left( n+1\right) \left(
n^{3}+7n^{2}-12n+24\right) }.
\end{eqnarray*}

The inequality (\ref{gn>}) follows by the Lemma \ref{L-p(x+1)>}.

\textbf{Step 6}: We show that $g_{2}\left( x\right) >0$ for $x\geq 10$.

Differentiation yields%
\begin{eqnarray*}
g_{2}^{\prime }\left( x\right) &=&-\frac{7}{256}\tfrac{\left(
896x^{7}+7346x^{6}+41\,033x^{5}-3438x^{4}-149\,813x^{3}-227\,064x^{2}-40%
\,824x+864\right) }{x^{7/8}\left( x^{4}+8x^{3}-5x^{2}+12x+24\right) ^{2}} \\
&&+\frac{3\pi }{128\Gamma \left( 7/8\right) }\tfrac{\left(
55x^{6}+3806x^{5}+17\,101x^{4}+216x^{3}-71\,514x^{2}-73\,824x-33\,840\right) 
}{\left( x^{4}+8x^{3}-5x^{2}+12x+24\right) ^{2}} \\
&&+\frac{2}{2x+1}+\tfrac{11x^{6}+16x^{5}+182x^{4}+72x^{3}-993x^{2}-384x-504}{%
\left( x^{4}+8x^{3}-5x^{2}+12x+24\right) ^{2}}.
\end{eqnarray*}

Application of the second inequality of (\ref{L-gi<>}) in Lemma \ref{L-g(x)<}
gives%
\begin{equation*}
\Gamma \left( \frac{7}{8}\right) <\left[ \frac{x^{2}+2}{\left( x+2\right) x}%
\right] _{x=7/8}=\frac{177}{161}<\frac{6}{5}
\end{equation*}%
and since $\pi >3$, it is acquired that%
\begin{eqnarray*}
g_{2}^{\prime }\left( x\right) &>&-\frac{7}{256}\tfrac{\left(
896x^{7}+7346x^{6}+41\,033x^{5}-3438x^{4}-149\,813x^{3}-227\,064x^{2}-40%
\,824x+864\right) }{x^{7/8}\left( x^{4}+8x^{3}-5x^{2}+12x+24\right) ^{2}} \\
&&+\frac{3\times 3}{128\times 6/5}\tfrac{\left(
55x^{6}+3806x^{5}+17\,101x^{4}+216x^{3}-71\,514x^{2}-73\,824x-33\,840\right) 
}{\left( x^{4}+8x^{3}-5x^{2}+12x+24\right) ^{2}} \\
&&+\frac{2}{2x+1}+\tfrac{11x^{6}+16x^{5}+182x^{4}+72x^{3}-993x^{2}-384x-504}{%
\left( x^{4}+8x^{3}-5x^{2}+12x+24\right) ^{2}} \\
&=&\tfrac{1}{256\left( x^{4}+8x^{3}-5x^{2}+12x+24\right) ^{2}}\left( \tfrac{%
g_{4}\left( x\right) }{\left( 2x+1\right) }-7\tfrac{g_{3}\left( x\right) }{%
x^{7/8}}\right) ,
\end{eqnarray*}%
where%
\begin{eqnarray*}
g_{3}\left( x\right) &=&896x^{7}+7346x^{6}+41\,033x^{5}-3438x^{4} \\
&&-149\,813x^{3}-227\,064x^{2}-40\,824x+864,
\end{eqnarray*}%
\begin{eqnarray*}
g_{4}\left( x\right)
&=&512x^{8}+15\,474x^{7}+153\,661x^{6}+638\,728x^{5}+482\,131x^{4} \\
&&-2496\,996x^{3}-3787\,398x^{2}-2184\,000x-341\,712.
\end{eqnarray*}

Making a change of variable $t=x-3\geq 7$ yields%
\begin{equation*}
\begin{array}{c}
\bigskip g_{3}\left( x\right)
=896t^{7}+26\,162t^{6}+342\,605t^{5}+2450\,487t^{4}+10\,008\,901t^{3} \\ 
+22\,815\,555t^{2}+26\,081\,658t+10\,797\,192>0,%
\end{array}%
\end{equation*}%
\begin{equation*}
\begin{array}{r}
\bigskip g_{4}\left( x\right)
=512t^{8}+27\,762t^{7}+607\,639t^{6}+7103\,356t^{5}+48\,333\,256t^{4} \\ 
+194\,587\,122t^{3}+448\,344\,153t^{2}+530\,387\,220t+235\,084\,068>0.%
\end{array}%
\end{equation*}%
Thus, to prove that $g_{2}^{\prime }\left( x\right) >0$ for $x\geq 10$, it
suffices to prove that%
\begin{equation*}
g_{5}\left( x\right) :=\ln \frac{g_{4}\left( x\right) }{2x+1}-\ln \left( 7%
\frac{g_{3}\left( x\right) }{x^{7/8}}\right) >0.
\end{equation*}

Differentiation again leads to%
\begin{eqnarray*}
g_{5}^{\prime }\left( x\right) &=&\frac{g_{4}^{\prime }\left( x\right) }{%
g_{4}\left( x\right) }-\frac{2}{2x+1}-\frac{g_{3}^{\prime }\left( x\right) }{%
g_{3}\left( x\right) }+\frac{7}{8}\frac{1}{x} \\
&=&\frac{1}{8}\frac{x^{4}+8x^{3}-5x^{2}+12x+24}{x\left( 2x+1\right) }\frac{%
g_{6}\left( x\right) }{g_{3}\left( x\right) g_{4}\left( x\right) },
\end{eqnarray*}%
where%
\begin{equation*}
\begin{array}{l}
g_{6}\left( x\right)
=6422\,528x^{12}+40\,606\,208x^{11}-29\,604\,936x^{10}-195\,118\,044x^{9}%
\bigskip \\ 
\multicolumn{1}{r}{-8157\,468\,886x^{8}-54\,727\,744\,833x^{7}-28\,074\,816%
\,632x^{6}\bigskip} \\ 
\multicolumn{1}{r}{-33\,746\,602\,635x^{5}-132\,036\,870\,576x^{4}-76\,358%
\,742\,474x^{3}\bigskip} \\ 
\multicolumn{1}{r}{-16\,962\,249\,960x^{2}-1692\,953\,136x-86\,111\,424.}%
\end{array}%
\end{equation*}%
Lemma \ref{L-P-zp} implies that the polynomial $g_{6}\left( x\right) $ has a
unique zero point $x_{0}\in \left( 0,\infty \right) $ such that $g_{6}\left(
x\right) <0$ for $x\in \left( 0,x_{0}\right) $ and $g_{6}\left( x\right) >0$
for $x\in \left( x_{0},\infty \right) $. This in combination with $%
g_{6}\left( 7\right) =56\,640\,373\,211\,408\,308>0$ indicates that $%
g_{6}\left( x\right) >0$ for $x\geq 7$. Therefore, $g_{5}^{\prime }\left(
x\right) >0$ for $x\geq 7$, and so%
\begin{equation*}
g_{5}\left( x\right) \geq g_{5}\left( 7\right) =\ln 44\,608\,161\,668-\frac{1%
}{8}\ln 7-\ln 15-\ln 2220\,734\,176=0.048\,79>0,
\end{equation*}%
which reveals that $g_{2}^{\prime }\left( x\right) >0$ for $x\geq 10$. It
follows that 
\begin{equation*}
g_{2}\left( x\right) \geq g_{2}\left( 10\right) =\gamma +\ln \tfrac{21}{2}-%
\tfrac{516\,789}{282\,304}\tfrac{\pi }{\Gamma \left( 7/8\right) }+\tfrac{%
649\,299}{282\,304}\sqrt[8]{10}-\tfrac{6005}{8822}\approx 0.037141>0.
\end{equation*}

\textbf{Step 7}: Steps 6 and 5 show that $g\left( n\right) >0$ for $n\geq 10$%
, which in conjunction with Step 3 yield%
\begin{equation*}
D_{n}=\frac{8}{7}d_{n+1}-d_{n}>g\left( n\right) >0,
\end{equation*}%
that is, $d_{n+1}>\left( 7/8\right) d_{n}$ for $n\geq 10$. Taking into
account Step 2, we conclude that $d_{n}=0$ for $n=0,1,2,3$ and $d_{n}>0$ for 
$n\geq 4$, and therefore, $f_{7}\left( r\right) =f_{5}\left( r\right)
-r^{\prime -7/4}f_{6}\left( r\right) >0$ for $r\in \left( 0,1\right) $. Thus
the function $f_{1}/f_{2}$, that is, $F$, is strictly decreasing on $\left(
0,1\right) $.

It follows from the decreasing property of the function $F$ on $\left(
0,1\right) $ that%
\begin{equation*}
\frac{11\left( \pi -2\right) }{4\pi }=\lim_{r\rightarrow 1^{-}}F\left(
r\right) <F\left( r\right) <\lim_{r\rightarrow 0^{+}}F\left( r\right) =1,
\end{equation*}%
which implies inequalities (\ref{MI-7/4-1}) and the first and second ones in
(\ref{MI-7/4-2}). The third one in (\ref{MI-7/4-2}) is equivalent to%
\begin{equation*}
\tfrac{22-7\pi }{4\pi }+\tfrac{11\left( \pi -2\right) }{4\pi }%
S_{11/4,7/4}\left( 1,r^{\prime }\right) -\tfrac{22}{7\pi }S_{11/4,7/4}\left(
1,r^{\prime }\right) =-\tfrac{11\left( 22-7\pi \right) }{28\pi }\left(
S_{11/4,7/4}\left( 1,r^{\prime }\right) -\tfrac{7}{11}\right) <0,
\end{equation*}%
where the inequality holds due to the increasing property of Stolarsky means
in their variables (P1) and $r^{\prime }\in \left( 0,1\right) $.

Since%
\begin{equation*}
\lim_{r\rightarrow 0^{+}}\frac{\left( 2/\pi \right) E\left( r\right) }{%
S_{11/4,7/4}\left( 1,r^{\prime }\right) }=1\text{ \ and \ }%
\lim_{r\rightarrow 1^{-}}\frac{\left( 2/\pi \right) E\left( r\right) }{%
S_{11/4,7/4}\left( 1,r^{\prime }\right) }=\frac{22}{7\pi },
\end{equation*}%
the coefficients $1$ and $22/\left( 7\pi \right) $ are the best.

Thus we complete the proof.
\end{proof}

\begin{remark}
From the Step 7 in previous proof, we have $d_{n+1}>\left( 7/8\right) d_{n}$
for $n\geq 10$. This is also valid for $4\leq n\leq 9$ by an easy
verification. It is derived that $d_{n}>\left( 7/8\right) ^{n-4}d_{4}$ for $%
n\geq 5$, and so we get%
\begin{equation*}
\sum_{n=0}^{\infty }d_{n}r^{2n}=\sum_{n=4}^{\infty
}d_{n}r^{2n}>d_{4}\sum_{n=4}^{\infty }\left( \frac{7}{8}\right) ^{n-4}r^{2n}=%
\frac{35}{4096}\frac{r^{8}}{8-7r^{2}}.
\end{equation*}%
Thus we obtain that%
\begin{equation*}
f_{7}\left( r\right) =f_{5}\left( r\right) -r^{\prime -7/4}f_{6}\left(
r\right) =\frac{3\pi }{16}r^{6}\sum_{n=0}^{\infty }d_{n}r^{2n}>\frac{105\pi 
}{2^{16}}\frac{r^{14}}{8-7r^{2}}.
\end{equation*}
\end{remark}

\section{Proof of Theorem \protect\ref{MT-2}}

Now we prove Theorems \ref{MT-2}.

\begin{proof}[Proof of Theorem \protect\ref{MT-2}]
Utilizing (\ref{Ee}) gives%
\begin{equation*}
1-\frac{2}{\pi }E\left( r\right) =-\sum_{n=1}^{\infty }\frac{\left( -\frac{1%
}{2}\right) _{n}\left( \frac{1}{2}\right) _{n}}{n!n!}r^{2n}:=\sum_{n=1}^{%
\infty }v_{n}r^{2n}.
\end{equation*}%
Applying binomial series to (\ref{S_5/2,2}) we have%
\begin{eqnarray*}
S_{5/2,2}\left( 1,r^{\prime }\right) &=&\left( \frac{4}{5}\frac{1-r^{\prime
5/2}}{1-r^{\prime 2}}\right) ^{2}=\frac{16}{25}\frac{\left( 1-r^{2}\right)
^{5/2}-2\left( 1-r^{2}\right) ^{5/4}+1}{r^{4}} \\
&=&\frac{16}{25}\frac{1}{r^{4}}\left( \sum_{n=0}^{\infty }\frac{\left( -%
\frac{5}{2}\right) _{n}}{n!}r^{2n}-2\sum_{n=0}^{\infty }\frac{\left( -\frac{5%
}{4}\right) _{n}}{n!}r^{2n}+1\right) \\
&=&1-\frac{16}{25}\sum_{n=1}^{\infty }\frac{2\left( -\frac{5}{4}\right)
_{n+2}-\left( -\frac{5}{2}\right) _{n+2}}{\left( n+2\right) !}r^{2n},
\end{eqnarray*}%
which implies that%
\begin{equation*}
1-S_{5/2,2}\left( 1,r^{\prime }\right) =\frac{16}{25}\sum_{n=1}^{\infty }%
\frac{2\left( -\frac{5}{4}\right) _{n+2}-\left( -\frac{5}{2}\right) _{n+2}}{%
\left( n+2\right) !}r^{2n}:=\sum_{n=1}^{\infty }u_{n}r^{2n}.
\end{equation*}%
Thus, the function $G\left( r\right) $ can be expressed as%
\begin{equation*}
G\left( r\right) =\frac{1-\left( 2/\pi \right) E\left( r\right) }{%
1-S_{5/2,2}\left( 1,r^{\prime }\right) }=\frac{\sum_{n=1}^{\infty
}v_{n}r^{2n}}{\sum_{n=1}^{\infty }u_{n}r^{2n}},
\end{equation*}%
where, by the formula $\left( a\right) _{n}=a\left( a+1\right) _{n-1}$, $%
v_{n}$ and $u_{n}$ can be written as%
\begin{eqnarray*}
v_{n} &=&\frac{1}{2}\frac{1}{n!n!}\left( \frac{1}{2}\right) _{n-1}\left( 
\frac{1}{2}\right) _{n}, \\
u_{n} &=&\frac{6}{5}\frac{1}{\left( n+2\right) !}\left( \frac{1}{2}\right)
_{n-1}+\frac{2}{5}\frac{1}{\left( n+2\right) !}\left( \frac{3}{4}\right)
_{n}.
\end{eqnarray*}

By Lemma \ref{L-ps-A/B}, due to $u_{n}>0$ for $n\geq 1$, to prove the
function $G$ is increasing on $\left( 0,1\right) $, it suffices to prove the
sequence $\{v_{n}/u_{n}\}$ is increasing for $n\geq 1$. In view of $v_{n}>0$
for $n\geq 1$, which suffices to check that $\left( v_{n+1}/v_{n}\right)
u_{n}-u_{n+1}>0$.

A simple verification yields%
\begin{equation*}
\frac{v_{n+1}}{v_{n}}=\frac{\left( n-1/2\right) \left( n+1/2\right) }{\left(
n+1\right) ^{2}},
\end{equation*}%
and%
\begin{equation*}
\frac{v_{n+1}}{v_{n}}u_{n}-u_{n+1}=\frac{3}{5}\frac{\left( 3n+1\right) }{%
\left( n+1\right) ^{2}\left( n+3\right) !}\left( \frac{1}{2}\right) _{n}+%
\frac{1}{10}\frac{n^{2}-11n-6}{\left( n+1\right) ^{2}\left( n+3\right) !}%
\left( \frac{3}{4}\right) _{n}.
\end{equation*}%
Direct computations give%
\begin{equation*}
\begin{array}{r}
\frac{v_{n+1}}{v_{n}}u_{n}-u_{n+1}=0,0,\frac{3}{81\,920},\frac{21}{512\,000},%
\frac{47}{1310\,720},\frac{1881}{64\,225\,280},\frac{157\,531}{6710\,886\,400%
},\bigskip \\ 
\frac{42\,559}{2264\,924\,160},\frac{507\,577}{33\,554\,432\,000},\frac{%
997\,177}{81\,201\,725\,440},\frac{20\,743\,573}{2061\,584\,302\,080}%
\end{array}%
\end{equation*}%
for $n=1$, $2$, ..., $11$, respectively. And, for $n\geq 12$, it is evident
that $\left( v_{n+1}/v_{n}\right) u_{n}-u_{n+1}>0$ due to $n^{2}-11n-6>0$.

Consequently, we obtain%
\begin{equation*}
1=\lim_{r\rightarrow 0^{+}}G\left( r\right) <G\left( r\right)
<\lim_{r\rightarrow 1^{-}}G\left( r\right) =\frac{25\left( \pi -2\right) }{%
9\pi },
\end{equation*}%
which implies inequalities (\ref{MI-2-1}) and the first and second ones in (%
\ref{MI-2-2}). The third one in (\ref{MI-2-2}) is equivalent to%
\begin{equation*}
-\tfrac{2\left( 8\pi -25\right) }{9\pi }+\tfrac{25\left( \pi -2\right) }{%
9\pi }S_{2,5/2}\left( 1,r^{\prime }\right) -\tfrac{25}{8\pi }S_{2,5/2}\left(
1,r^{\prime }\right) =\tfrac{25\left( 8\pi -25\right) }{72\pi }\left(
S_{2,5/2}\left( 1,r^{\prime }\right) -\tfrac{16}{25}\right) >0,
\end{equation*}%
where the inequality holds due to the increasing property of Stolarsky means
in their variables (P1) and $r^{\prime }\in \left( 0,1\right) $.

Since 
\begin{equation*}
\lim_{r\rightarrow 0^{+}}\frac{\left( 2/\pi \right) E\left( r\right) }{%
S_{5/2,2}\left( 1,r^{\prime }\right) }=1\text{ \ and \ }\lim_{r\rightarrow
1^{-}}\frac{\left( 2/\pi \right) E\left( r\right) }{S_{5/2,2}\left(
1,r^{\prime }\right) }=\frac{25}{8\pi },
\end{equation*}%
the coefficients $1$ and $25/\left( 8\pi \right) $ are the best.

This completes the proof.
\end{proof}

\begin{remark}
Denote by $w_{n}=v_{n}-u_{n}$. It is easy to verify that the relation%
\begin{equation*}
w_{n+1}-\frac{v_{n+1}}{v_{n}}w_{n}=\frac{v_{n+1}}{v_{n}}u_{n}-u_{n+1}
\end{equation*}%
holds for $n\geq 0$. From the proof of Theorem \ref{MT-2} we clearly see that%
\begin{equation}
w_{n+1}-\frac{v_{n+1}}{v_{n}}w_{n}>0  \label{w_n/v_n}
\end{equation}%
for $n\geq 3$, which, due to $w_{1}=w_{2}=w_{3}=0$, $w_{4}=3\times
2^{-14}/5>0$, means that $w_{n}>0$ for $n\geq 4$. This also yields%
\begin{equation*}
S_{5/2,2}\left( 1,r^{\prime }\right) -\frac{2}{\pi }E\left( r\right)
=\sum_{n=1}^{\infty }v_{n}r^{2n}-\sum_{n=1}^{\infty
}u_{n}r^{2n}=r^{8}\sum_{n=4}^{\infty }w_{n}r^{2n-8},
\end{equation*}%
so we have
\end{remark}

\begin{proposition}
The function%
\begin{equation*}
G_{1}\left( r\right) =\frac{S_{5/2,2}\left( 1,r^{\prime }\right) -\frac{2}{%
\pi }E\left( r\right) }{r^{8}}
\end{equation*}%
is convex and strictly increasing from $\left( 0,1\right) $ onto $\left(
\left( 3/5\right) 2^{-14},16/25-2/\pi \right) $. Consequently, we have%
\begin{equation*}
\frac{3}{5}\times 2^{-14}<\frac{S_{5/2,2}\left( 1,r^{\prime }\right) -\left(
2/\pi \right) E\left( r\right) }{r^{8}}<\frac{16}{25}-\frac{2}{\pi }\approx
0.003802.
\end{equation*}
\end{proposition}

\begin{remark}
Further, the relation (\ref{w_n/v_n}) also indicates that for $n\geq 4$%
\begin{equation*}
w_{n}>\frac{w_{4}}{v_{4}}v_{n}=\frac{3\times 2^{-14}/5}{175\times 2^{-14}}%
v_{n}=\frac{3}{875}v_{n},
\end{equation*}%
and therefore,%
\begin{eqnarray*}
S_{5/2,2}\left( 1,r^{\prime }\right) -\frac{2}{\pi }E\left( r\right)
&=&\sum_{n=4}^{\infty }w_{n}r^{2n}>\frac{3}{875}\sum_{n=4}^{\infty
}v_{n}r^{2n}=\frac{3}{875}\left( \sum_{n=1}^{\infty
}v_{n}r^{2n}-\sum_{n=1}^{3}v_{n}r^{2n}\right) \\
&=&\frac{3}{875}\left( 1-\frac{2}{\pi }E\left( r\right) -\left( \frac{1}{4}%
r^{2}+\frac{3}{64}r^{4}+\frac{5}{256}r^{6}\right) \right) ,
\end{eqnarray*}

which can be stated as a proposition as follows:
\end{remark}

\begin{proposition}
For $r\in \left( 0,1\right) $, it holds that%
\begin{equation*}
\frac{2}{\pi }E\left( r\right) <\frac{875}{872}S_{5/2,2}\left( 1,r^{\prime
}\right) -\frac{3}{872}\left( 1-\frac{1}{4}r^{2}-\frac{3}{64}r^{4}-\frac{5}{%
256}r^{6}\right) .
\end{equation*}
\end{proposition}

\section{Corollaries}

As direct consequences of Theorems \ref{MT-7/4} and \ref{MT-2}, we have

\begin{corollary}
\label{MC-1}Both the functions%
\begin{equation*}
r\mapsto \frac{2}{\pi }E\left( r\right) -S_{11/4,7/4}\left( 1,r^{\prime
}\right) \text{ \ and \ }r\mapsto \frac{\left( 2/\pi \right) E\left(
r\right) }{S_{11/4,7/4}\left( 1,r^{\prime }\right) }
\end{equation*}%
are strictly increasing from $\left( 0,1\right) $ onto $\left( 0,2/\pi
-7/11\right) $ and $\left( 1,22/\left( 7\pi \right) \right) $, respectively.
And therefore, we have%
\begin{eqnarray}
0 &<&\frac{2}{\pi }E\left( r\right) -S_{11/4,7/4}\left( 1,r^{\prime }\right)
<\frac{2}{\pi }-\frac{7}{11}\approx 0.00025614,  \label{Bi-E-S_7/4} \\
1 &<&\frac{\left( 2/\pi \right) E\left( r\right) }{S_{11/4,7/4}\left(
1,r^{\prime }\right) }<\frac{22}{7\pi }\approx 1.0004.  \label{Bi-E/S_7/4}
\end{eqnarray}
\end{corollary}

\begin{proof}
By the increasing property of Stolarsky means in their variables and $%
r^{\prime }=\sqrt{1-r^{2}}$, it is seen that $S_{11/4,7/4}\left( 1,r^{\prime
}\right) $ is decreasing with respect to $r$ on $\left( 0,1\right) $, and so
both the functions 
\begin{equation*}
r\mapsto \left( 1-S_{11/4,7/4}\left( 1,r^{\prime }\right) \right) \text{ \
and \ }r\mapsto \frac{1}{S_{11/4,7/4}\left( 1,r^{\prime }\right) }
\end{equation*}%
are positive and increasing on $\left( 0,1\right) $.

From Theorem \ref{MT-7/4}, we see that $r\mapsto \left( 1-F\left( r\right)
\right) $ is positive and strictly increasing on $\left( 0,1\right) $. It
follows from the identity%
\begin{equation*}
\frac{2}{\pi }E\left( r\right) -S_{11/4,7/4}\left( 1,r^{\prime }\right)
=\left( 1-F\left( r\right) \right) \left( 1-S_{11/4,7/4}\left( 1,r^{\prime
}\right) \right)
\end{equation*}%
that the function $r\mapsto \left( 2/\pi \right) E\left( r\right)
-S_{11/4,7/4}\left( 1,r^{\prime }\right) $ is also positive and strictly
increasing on $\left( 0,1\right) $. Therefore, the double inequality (\ref%
{Bi-E-S_7/4}) is valid.

Making use of the assertion proved previously, and noting that%
\begin{equation*}
\frac{\left( 2/\pi \right) E\left( r\right) }{S_{11/4,7/4}\left( 1,r^{\prime
}\right) }=1+\left( \frac{2}{\pi }E\left( r\right) -S_{11/4,7/4}\left(
1,r^{\prime }\right) \right) \times \frac{1}{S_{11/4,7/4}\left( 1,r^{\prime
}\right) }
\end{equation*}%
gives the desired assertion. Then the estimate inequalities (\ref{Bi-E/S_7/4}%
) follow.
\end{proof}

Using the same technique\ we can prove

\begin{corollary}
\label{MC-2}Both the functions%
\begin{equation*}
r\mapsto \frac{2}{\pi }E\left( r\right) -S_{5/2,2}\left( 1,r^{\prime
}\right) \text{ \ and \ }r\mapsto \frac{\left( 2/\pi \right) E\left(
r\right) }{S_{5/2,2}\left( 1,r^{\prime }\right) }
\end{equation*}%
are strictly decreasing from $\left( 0,1\right) $ onto $\left( 2/\pi
-16/25,0\right) $ and $\left( 25/\left( 8\pi \right) ,1\right) $. And
therefore, we have%
\begin{eqnarray}
-0.0033802 &\approx &\frac{2}{\pi }-\frac{16}{25}<\frac{2}{\pi }E\left(
r\right) -S_{2,5/2}\left( 1,r^{\prime }\right) <0,  \label{Bi-E-S_2} \\
0.99472 &\approx &\frac{25}{8\pi }<\frac{\left( 2/\pi \right) E\left(
r\right) }{S_{5/2,2}\left( 1,r^{\prime }\right) }<1.  \label{Bi-E/S_2}
\end{eqnarray}
\end{corollary}

Now we give a monotonicity property for the ratio $\left( 2/\pi \right)
E\left( r\right) /S_{9/2-p,p}\left( 1,r^{\prime }\right) $ with respect to $%
r $ for $p\in (-\infty ,9/4]$. To this end, we need a known statement proved
in \cite[Theorem 5]{Losonczi-PMD-75(1-2)-2009} by Losonczi (see also \cite[%
Theorem 3.4]{Yang-IJMMS-591382-2009}).

\begin{lemma}
\label{RS_p-m}For fixed $c>0$, $0<x<y<z$, the function%
\begin{equation*}
p\mapsto \frac{S_{2c-p,p}\left( x,y\right) }{S_{2c-p,p}\left( x,z\right) }%
:=R_{2c-p,p}\left( x,y,z\right)
\end{equation*}%
is strictly decreasing on $(-\infty ,c]$ and strictly increasing on $%
[c,\infty )$.
\end{lemma}

\begin{corollary}
\label{MC-E/S_p-mon.}Let $p\in (-\infty ,9/4]$. The the function%
\begin{equation*}
r\mapsto \frac{\left( 2/\pi \right) E\left( r\right) }{S_{9/2-p,p}\left(
1,r^{\prime }\right) }:=R_{p}\left( r\right)
\end{equation*}%
is strictly increasing from $\left( 0,1\right) $ onto $\left( 1,2/\left( \pi
\theta _{p}\right) \right) $ if and only if $p\in (-\infty ,7/4]$ and
strictly decreasing from $\left( 0,1\right) $ onto $\left( 2/\left( \pi
\theta _{p}\right) ,1\right) $ if $p\in \lbrack 2,9/4]$. Consequently, we
have%
\begin{eqnarray}
S_{9/2-p,p}\left( 1,r^{\prime }\right) &<&\frac{2}{\pi }E\left( r\right) <%
\frac{2}{\pi \theta _{p}}S_{9/2-p,p}\left( 1,r^{\prime }\right) \text{ if }%
p\in (0,\frac{7}{4}],  \label{E/S_p-p<7/4} \\
\frac{2}{\pi \theta _{p}}S_{9/2-p,p}\left( 1,r^{\prime }\right) &<&\frac{2}{%
\pi }E\left( r\right) <S_{9/2-p,p}\left( 1,r^{\prime }\right) \text{ if }%
p\in \lbrack 2,\frac{9}{4}],  \label{E/S_p-p>2}
\end{eqnarray}%
where the coefficients $1$ and $2/\left( \pi \theta _{p}\right) $ are the
best possible, here $\theta _{p}$ is defined by (\ref{theta_p(c)}).
\end{corollary}

\begin{proof}
(i) The necessity can be derived from%
\begin{equation*}
\lim_{r\rightarrow 0^{+}}\frac{d\left( \ln R_{p}\left( r\right) \right) /dr}{%
8r^{7}}\geq 0.
\end{equation*}%
From $\lim_{r\rightarrow 0^{+}}R_{p}\left( r\right) =\lim_{r\rightarrow
0^{+}}S_{9/2-p,p}\left( 1,r^{\prime }\right) =1$ and L'Hospitial rule it is
obtained that%
\begin{eqnarray*}
1 &=&\lim_{r\rightarrow 0^{+}}\left( \frac{\ln R_{p}\left( r\right) }{%
R_{p}\left( r\right) -1}\frac{1}{S_{9/2-p,p}\left( 1,r^{\prime }\right) }%
\right) \\
&=&\lim_{r\rightarrow 0^{+}}\frac{\ln R_{p}\left( r\right) }{\left( 2/\pi
\right) E\left( r\right) -S_{9/2-p,p}\left( 1,r^{\prime }\right) }%
=\lim_{r\rightarrow 0^{+}}\frac{d\left( \ln R_{p}\left( r\right) \right) /dr%
}{d\left( \left( 2/\pi \right) E\left( r\right) -S_{9/2-p,p}\left(
1,r^{\prime }\right) \right) /dr}.
\end{eqnarray*}%
By the expansion (\ref{E-S_p}) and L'Hospitial rule we have%
\begin{eqnarray*}
\frac{\left( 4p-7\right) \left( 4p-11\right) }{5\times 2^{14}}
&=&\lim_{r\rightarrow 0^{+}}\frac{\left( 2/\pi \right) E\left( r\right)
-S_{9/2-p,p}\left( 1,r^{\prime }\right) }{r^{8}} \\
&=&\lim_{r\rightarrow 0^{+}}\frac{d\left( \left( 2/\pi \right) E\left(
r\right) -S_{9/2-p,p}\left( 1,r^{\prime }\right) \right) /dr}{8r^{7}}.
\end{eqnarray*}%
It follows that%
\begin{eqnarray*}
&&\lim_{r\rightarrow 0^{+}}\frac{d\left( \ln R_{p}\left( r\right) \right) /dr%
}{8r^{7}} \\
&=&\lim_{r\rightarrow 0^{+}}\left( \frac{d\left( \ln R_{p}\left( r\right)
\right) /dr}{d\left( \left( 2/\pi \right) E\left( r\right)
-S_{9/2-p,p}\left( 1,r^{\prime }\right) \right) /dr}\frac{d\left( \left(
2/\pi \right) E\left( r\right) -S_{9/2-p,p}\left( 1,r^{\prime }\right)
\right) /dr}{8r^{7}}\right) \\
&=&\frac{\left( 4p-7\right) \left( 4p-11\right) }{5\times 2^{14}},
\end{eqnarray*}%
which together with $p\in (-\infty ,9/4]$ gives the necessary condition $%
p\in (-\infty ,7/4]$.

(ii) To prove that the condition $p\in (-\infty ,7/4]$ is necessary, we have
to prove that the function%
\begin{equation*}
r\mapsto \frac{S_{9/2-p_{0},p_{0}}\left( 1,r^{\prime }\right) }{%
S_{9/2-p,p}\left( 1,r^{\prime }\right) }:=R_{p_{0},p}^{\ast }\left(
r^{\prime }\right)
\end{equation*}%
is strictly increasing (decreasing) on $\left( 0,1\right) $ if $p<\left(
>\right) p_{0}$. Due to $r^{\prime }=\sqrt{1-r^{2}}$, it suffices to prove
that $r^{\prime }\mapsto R_{p_{0},p}^{\ast }\left( r^{\prime }\right) $ is
strictly decreasing (increasing) on $\left( 0,1\right) $ if $p<\left(
>\right) p_{0}$. Assume that $r_{1}^{\prime }$, $r_{2}^{\prime }\in \left(
0,1\right) $ with $r_{1}^{\prime }<r_{2}^{\prime }$. Then $1<1/r_{2}^{\prime
}<1/r_{1}^{\prime }$. By Lemma \ref{RS_p-m}, we have%
\begin{equation*}
\frac{S_{2c-p_{0},p_{0}}\left( 1,1/r_{2}^{\prime }\right) }{%
S_{2c-p_{0},p_{0}}\left( 1,1/r_{1}^{\prime }\right) }<\left( >\right) \frac{%
S_{2c-p,p}\left( 1,1/r_{2}^{\prime }\right) }{S_{2c-p,p}\left(
1,1/r_{1}^{\prime }\right) },
\end{equation*}%
which is equivalent to%
\begin{equation*}
\frac{S_{2c-p_{0},p_{0}}\left( 1,1/r_{2}^{\prime }\right) }{S_{2c-p,p}\left(
1,1/r_{2}^{\prime }\right) }<\left( >\right) \frac{S_{2c-p_{0},p_{0}}\left(
1,1/r_{1}^{\prime }\right) }{S_{2c-p,p}\left( 1,1/r_{1}^{\prime }\right) }.
\end{equation*}%
From the homogeneity and symmetry of Stolarsky means with respect to their
variables, this shows that $r^{\prime }\mapsto R_{p_{0},p}^{\ast }\left(
r^{\prime }\right) $ is strictly decreasing (increasing) on $\left(
0,1\right) $ if $p<\left( >\right) p_{0}$.

Now, if $p_{0}=7/4$ and $p\in (-\infty ,7/4)$, then $r\mapsto
R_{p_{0},p}^{\ast }\left( r^{\prime }\right) $ is positive and strictly
increasing on $\left( 0,1\right) $. While Corollary \ref{MC-1} tells us that 
$r\mapsto R_{p_{0}}\left( r\right) $ is also. It is deduced by the relation 
\begin{equation*}
R_{p}\left( r\right) =\frac{\left( 2/\pi \right) E\left( r\right) }{%
S_{9/2-p_{0},p_{0}}\left( 1,r^{\prime }\right) }\times \frac{%
S_{9/2-p_{0},p_{0}}\left( 1,r^{\prime }\right) }{S_{9/2-p,p}\left(
1,r^{\prime }\right) }:=R_{p_{0}}\left( r\right) \times R_{p_{0},p}^{\ast
}\left( r^{\prime }\right)
\end{equation*}%
that $r\mapsto R_{p}\left( r\right) $ is also strictly increasing on $\left(
0,1\right) $.

(iii) We continue to show that the function $r\mapsto R_{p}\left( r\right) $
is strictly decreasing if $p\in \lbrack 2,9/4]$. Let $p_{0}=2$ and $p\in
(2,9/4]$. Then $r\mapsto R_{p_{0},p}^{\ast }\left( r^{\prime }\right) $ is
positive and strictly decreasing on $\left( 0,1\right) $. Similarly, this
together with Corollary \ref{MC-2} reveals that $r\mapsto R_{p}\left(
r\right) $ is also strictly decreasing on $\left( 0,1\right) $.

(iv) Lastly, the inequalities (\ref{E/S_p-p<7/4}) and (\ref{E/S_p-p>2})
follow from the monotonicity of the function $r\mapsto R_{p}\left( r\right) $
on $\left( 0,1\right) $.

The proof is finished.
\end{proof}

\begin{remark}
Let $p=9/8,3/2,7/4$; $2,9/4$ in Corollary \ref{MC-E/S_p-mon.}. Then by the
monotonicity of $S_{2c-p,p}$ and $\left( 1/\theta _{p}\right) S_{2c-p,p}$ in 
$p$ on $\left( 0,c\right) $ given in (P2) and (P3) we have%
\begin{eqnarray*}
He_{9/4}\left( 1,r^{\prime }\right) &<&A_{3/2}\left( 1,r^{\prime }\right)
<S_{11/4,7/4}\left( 1,r^{\prime }\right) <\frac{2}{\pi }E\left( r\right) \\
&<&\frac{22}{7\pi }S_{11/4,7/4}\left( 1,r^{\prime }\right) <\frac{2^{5/3}}{%
\pi }A_{3/2}\left( 1,r^{\prime }\right) <\frac{2\times 3^{4/9}}{\pi }%
He_{9/4}\left( 1,r^{\prime }\right) ,
\end{eqnarray*}%
\begin{equation*}
\frac{2e^{4/9}}{\pi }I_{9/4}\left( 1,r^{\prime }\right) <\frac{25}{8\pi }%
S_{5/2,2}\left( 1,r^{\prime }\right) <\frac{2}{\pi }E\left( r\right)
<S_{5/2,2}\left( 1,r^{\prime }\right) <I_{9/4}\left( 1,r^{\prime }\right) ,
\end{equation*}%
where $He_{p}\left( a,b\right) =He\left( a^{p},b^{p}\right) ^{1/p}$ and $%
I_{p}\left( a,b\right) =I\left( a^{p},b^{p}\right) ^{1/p}$ are the $p$-order
Heronian mean and identric (exponential) mean of positive numbers $a$ and $b$%
, respectively.
\end{remark}

Using expansion (\ref{E-S_p}) and Corollaries \ref{MC-1}, \ref{MC-2}
together with the property of Stolarsky means (P2), we obtain immediately

\begin{corollary}
\label{MC-E><S_p}For $p\in (-\infty ,9/4]$, the inequality%
\begin{equation}
\frac{2}{\pi }E\left( r\right) >S_{9/2-p,p}\left( 1,r^{\prime }\right)
\label{S_p,9/2,p-E}
\end{equation}%
holds for all $r\in \left( 0,1\right) $ if and only $p\in (-\infty ,7/4]$.
The inequality (\ref{S_p,9/2,p-E}) reverses for $p\in \lbrack 2,9/4]$.
\end{corollary}

\begin{remark}
Taking $p=-9/2,-9/4,0,9/8,3/2,7/4$; $2,9/4$ in Corollary \ref{MC-E><S_p}, we
get immediately%
\begin{eqnarray}
A_{9/2}^{1/3}\left( 1,r^{\prime }\right) G^{2/3}\left( 1,r^{\prime }\right)
&<&\sqrt{G\left( 1,r^{\prime }\right) He_{9/2}\left( 1,r^{\prime }\right) }%
<L_{9/2}\left( 1,r^{\prime }\right) <He_{9/4}\left( 1,r^{\prime }\right) < 
\notag \\
A_{3/2}\left( 1,r^{\prime }\right) &<&S_{11/4,7/4}\left( 1,r^{\prime
}\right) <\frac{2}{\pi }E\left( r\right) <S_{5/2,2}\left( 1,r^{\prime
}\right) <I_{9/4}\left( 1,r^{\prime }\right)  \label{E-S_p-Ic}
\end{eqnarray}%
holds for $r\in \left( 0,1\right) $, where $L_{p}\left( a,b\right) =L\left(
a^{p},b^{p}\right) ^{1/p}$ is the $p$-order logarithmic mean of positive
numbers $a$ and $b$.
\end{remark}

\begin{remark}
For $a,b>0$ with $a\neq b$, the Toader mean $\mathcal{T}(a,b)$ is defined in 
\cite{Toader-JMAA-218(2)-1998} by%
\begin{equation*}
\mathcal{T}(a,b)=\frac{2}{\pi }\int_{0}^{\pi /2}\sqrt{a^{2}\cos
^{2}t+b^{2}\sin ^{2}t}dt.
\end{equation*}%
An easy transformation yields%
\begin{equation*}
\mathcal{T}(a,b)=\left\{ 
\begin{array}{cc}
\frac{2}{\pi }aE\left( \sqrt{1-\left( b/a\right) ^{2}}\right) & \text{if }%
a>b, \\ 
\frac{2}{\pi }bE\left( \sqrt{1-\left( a/b\right) ^{2}}\right) & \text{if }%
a<b.%
\end{array}%
\right.
\end{equation*}%
Thus, all our results can be rewritten in the form of Toader mean, for
example, inequalities \ref{MI-7/4-2}, \ref{MI-2-2} and \ref{E-S_p-Ic} are
equivalent to%
\begin{equation*}
S_{11/4,7/4}\left( a,b\right) <\mathcal{T}\left( a,b\right) <\tfrac{22-7\pi 
}{4\pi }\max \left( a,b\right) +\tfrac{11\left( \pi -2\right) }{4\pi }%
S_{11/4,7/4}\left( a,b\right) <\tfrac{22}{7\pi }S_{11/4,7/4}\left(
a,b\right) ,
\end{equation*}%
\begin{equation*}
\tfrac{25}{8\pi }S_{2,5/2}\left( a,b\right) <-\tfrac{2\left( 8\pi -25\right) 
}{9\pi }\max \left( a,b\right) +\tfrac{25\left( \pi -2\right) }{9\pi }%
S_{2,5/2}\left( a,b\right) <\mathcal{T}\left( a,b\right) <S_{2,5/2}\left(
a,b\right) ,
\end{equation*}%
\begin{eqnarray*}
A_{9/2}^{1/3}\left( a,b\right) G^{2/3}\left( a,b\right) &<&\sqrt{G\left(
a,b\right) He_{9/2}\left( a,b\right) }<L_{9/2}\left( a,b\right)
<He_{9/4}\left( a,b\right) < \\
A_{3/2}\left( a,b\right) &<&S_{11/4,7/4}\left( a,b\right) <\mathcal{T}\left(
a,b\right) <S_{5/2,2}\left( a,b\right) <I_{9/4}\left( a,b\right) .
\end{eqnarray*}
\end{remark}

\section{Comparisons with some known approximations}

Let $\mathcal{A}\left( r\right) $ be the given approximation for $\left(
2/\pi \right) E\left( r\right) $ and let $\Delta \left( r\right) =\mathcal{A}%
\left( r\right) -\left( 2/\pi \right) E\left( r\right) $ denote the error.
In general, the most important criterion to measure of accuracy of the given
approximation $\mathcal{A}\left( r\right) $ for $\left( 2/\pi \right)
E\left( r\right) $ should be the maximum absolute error $\max_{r\in \left(
0,1\right) }|\Delta \left( r\right) |$. Due to $r\in \left( 0,1\right) $,
however, similar to Barnard et al.'s opinion in \cite{Barnard-JMA-32-2000},
if $\Delta \left( r\right) =\varepsilon _{n_{0}}r^{2n_{0}}+O\left(
r^{2n_{0}+2}\right) $ with $\varepsilon _{n_{0}}\neq 0$ by expanding in
Maclaurin series,\ then the leading item $\delta _{0}:=\varepsilon
_{n_{0}}r^{2n_{0}}$ can be viewed as a measure of accuracy of the given
approximation $\mathcal{A}\left( r\right) $ for $\left( 2/\pi \right)
E\left( r\right) $. For this reason, \emph{we call }$\mathcal{A}\left(
r\right) $\emph{\ an }$n_{0}$\emph{-order approximation for }$\left( 2/\pi
\right) E\left( r\right) $\emph{. }And, $\mathcal{A}\left( r\right) $\emph{\
is called an }$n_{0}$\emph{-order lower (upper) approximation for }$\left(
2/\pi \right) E\left( r\right) $ if $\Delta \left( r\right) \leq \left( \geq
\right) 0$ for all $r\in \left( 0,1\right) $. In most cases, the greater the
order $n_{0}$ is, the higher the accuracy of approximation $\mathcal{A}%
\left( r\right) $ for $\left( 2/\pi \right) E\left( r\right) $ is.

Of course, a desirable approximation $\mathcal{A}\left( r\right) $ also has
the simplicity of the expression. Unfortunately, there exists frequently,
certain negative correlation between the accuracy and simplicity.

Now we choose those approximations which have higher accuracy mentioned in
Introduction to compare with our ones.

\subsection{For some lower approximations}

In Tables 1, the values in the third column are derived by expanding in
power series, while $\max_{r\in \left( 0,1\right) }|\Delta _{i}\left(
r\right) |$ for $i=1$, $2$, $5$ in the fourth column are from \cite[Theorem
1.1]{Barnard-JMA-31-2000}, \cite[Theorem 2.5]{Wang-JAT-164-2012}, (\ref%
{Bi-E-S_7/4}), respectively.%
\begin{equation*}
\begin{array}{c}
\text{Table 1: \ The lower approximations} \\ 
\begin{tabular}{|c|c|c|c|}
\hline
$\mathcal{A}_{i}\left( r^{\prime }\right) $ & Expressions & $\delta
_{0}^{\left( i\right) }:=\varepsilon _{n_{0}}^{\left( i\right) }r^{2n_{0}}$
& $\max_{r\in \left( 0,1\right) }|\Delta _{i}\left( r\right) |$ \\ \hline
\multicolumn{1}{|l|}{$\mathcal{A}_{1}\left( r^{\prime }\right) $} & 
\multicolumn{1}{|l|}{$A_{3/2}\left( 1,r^{\prime }\right) $} & 
\multicolumn{1}{|l|}{$\frac{1}{2^{14}}r^{8}$} & \multicolumn{1}{|l|}{$=\frac{%
2}{\pi }-2^{-2/3}\approx 0.0066592$} \\ \hline
\multicolumn{1}{|l|}{$\mathcal{A}_{2}\left( r^{\prime }\right) $} & 
\multicolumn{1}{|l|}{$\frac{23A\left( 1,r^{\prime }\right) -5H\left(
1,r^{\prime }\right) -2S\left( 1,r^{\prime }\right) }{16}$} & 
\multicolumn{1}{|l|}{$\frac{3}{2^{20}}r^{12}$} & \multicolumn{1}{|l|}{$=%
\frac{2}{\pi }-\frac{23-2\sqrt{2}}{32}\approx 0.0062581$} \\ \hline
\multicolumn{1}{|l|}{$\mathcal{A}_{3}\left( r^{\prime }\right) $} & 
\multicolumn{1}{|l|}{$\frac{1}{128}\frac{\left( 9r^{\prime
}{}^{2}+14r^{\prime }+9\right) ^{2}}{\left( r^{\prime }+1\right) ^{3}}$} & 
\multicolumn{1}{|l|}{$\frac{1}{2^{20}}r^{12}$} & \multicolumn{1}{|l|}{$\geq 
\frac{2}{\pi }-\frac{81}{128}\approx 0.0038073$} \\ \hline
\multicolumn{1}{|l|}{$\mathcal{A}_{4}\left( r^{\prime }\right) $} & 
\multicolumn{1}{|l|}{$\frac{1+r^{\prime }+r^{\prime 2}}{2\left( 1+r^{\prime
}\right) }+\frac{1+r^{\prime }}{8}$} & \multicolumn{1}{|l|}{$\frac{263}{%
2^{16}}r^{8}$} & \multicolumn{1}{|l|}{$\geq \frac{2}{\pi }-\frac{5}{8}%
\approx 0.011620$} \\ \hline
\multicolumn{1}{|l|}{$\mathcal{A}_{5}\left( r^{\prime }\right) $} & 
\multicolumn{1}{|l|}{$S_{11/4,7/4}\left( 1,r^{\prime }\right) $} & 
\multicolumn{1}{|l|}{$\frac{1}{7\times 2^{21}}r^{12}$} & 
\multicolumn{1}{|l|}{$=\frac{2}{\pi }-\frac{7}{11}\approx 0.00025614$} \\ 
\hline
\end{tabular}%
\end{array}%
\end{equation*}%
Moreover, we can prove

\begin{lemma}
\label{A-53214}Let $\mathcal{A}_{i}\left( x\right) $ ($i=1,2,3,4,5$) be
defined on $\left( 0,\infty \right) $ by (\ref{E>A_3/2}), (\ref{Wang-1}), (%
\ref{Wang-2}), (\ref{HQ}) and (\ref{S_11/4,7/4}), respectively (see also
Table 1). Then the inequalities%
\begin{equation}
\mathcal{A}_{5}\left( x\right) >\mathcal{A}_{3}\left( x\right) >\mathcal{A}%
_{2}\left( x\right) >\mathcal{A}_{1}\left( x\right) >\mathcal{A}_{4}\left(
x\right)  \label{A_5>4>2>1>4}
\end{equation}%
hold for $x>0$ with $x\neq 1$.
\end{lemma}

\begin{proof}
(i) The first inequality in (\ref{A_5>4>2>1>4}) is equivalent to 
\begin{equation*}
\mathcal{A}_{5}\left( x^{4}\right) -\mathcal{A}_{3}\left( x^{4}\right) =%
\frac{7\left( x^{11}-1\right) }{11\left( x^{7}-1\right) }-\frac{\left(
9x^{8}+14x^{4}+9\right) ^{2}}{128\left( x^{4}+1\right) ^{3}}>0
\end{equation*}%
Factoring yields%
\begin{equation*}
\mathcal{A}_{5}\left( x^{4}\right) -\mathcal{A}_{3}\left( x^{4}\right) =%
\frac{1}{1408}\frac{\left( x-1\right) ^{5}}{\left( x^{7}-1\right) \left(
x^{4}+1\right) ^{3}}h_{1}\left( x\right) >0,
\end{equation*}%
where%
\begin{eqnarray*}
h_{1}\left( x\right)
&=&5x^{16}+35x^{15}+140x^{14}+420x^{13}+966x^{12}+1722x^{11}+2268x^{10}+2415x^{9}
\\
&&+2362x^{8}+2415x^{7}+2268x^{6}+1722x^{5}+966x^{4}+420x^{3}+140x^{2}+35x+5.
\end{eqnarray*}

(ii) A direct computation leads to the second inequality in (\ref%
{A_5>4>2>1>4}) is equivalent to%
\begin{eqnarray*}
\mathcal{A}_{3}\left( x\right) -\mathcal{A}_{2}\left( x\right) &=&\frac{1}{%
128}\frac{\left( 9x^{2}+14x+9\right) ^{2}}{\left( x+1\right) ^{3}}-\frac{23%
\frac{1+x}{2}-5\frac{2x}{1+x}-2\sqrt{\frac{1+x^{2}}{2}}}{16} \\
&=&\frac{\sqrt{2\left( 1+x^{2}\right) }}{16}-\frac{1}{128}\frac{%
11x^{4}+36x^{3}+34x^{2}+36x+11}{\left( x+1\right) ^{3}}>0,
\end{eqnarray*}%
where the inequality holds due to%
\begin{eqnarray*}
&&\left( \frac{\sqrt{2\left( 1+x^{2}\right) }}{16}\right) ^{2}-\left( \frac{1%
}{128}\frac{36x+34x^{2}+36x^{3}+11x^{4}+11}{\left( x+1\right) ^{3}}\right)
^{2} \\
&=&\frac{1}{2^{14}}\frac{\left( x-1\right) ^{6}\left( 7x^{2}+18x+7\right) }{%
\left( x+1\right) ^{6}}>0.
\end{eqnarray*}

(iii) It has been shown in \cite[Lemma 3.2]{Wang-JAT-164-2012} that the
third inequality in (\ref{A_5>4>2>1>4}) holds for $x>0$ with $x\neq 1$.

(iv) The last inequality in (\ref{A_5>4>2>1>4}) is equivalent to%
\begin{eqnarray*}
\mathcal{A}_{1}\left( x^{2}\right) ^{3}-\mathcal{A}_{4}\left( x^{2}\right)
^{3} &=&\left( \frac{1+x^{3}}{2}\right) ^{2}-\left( \frac{1+x^{2}+x^{4}}{%
2\left( 1+x^{2}\right) }+\frac{1+x^{2}}{8}\right) ^{3} \\
&=&\frac{\left( x-1\right) ^{6}}{512}\frac{\left(
3x^{6}+18x^{5}-3x^{4}+28x^{3}-3x^{2}+18x+3\right) }{\left( x^{2}+1\right)
^{3}}>0.
\end{eqnarray*}%
This completes the proof.
\end{proof}

\begin{remark}
In Table 1, both $\mathcal{A}_{1}\left( r^{\prime }\right) $ and $\mathcal{A}%
_{4}\left( r^{\prime }\right) $ are $4$-order lower approximations. And for
the accuracy, since $\varepsilon _{4}^{\left( 1\right) }<\varepsilon
_{4}^{\left( 4\right) }$ and $\max_{r\in \left( 0,1\right) }|\Delta
_{1}\left( r\right) |<\max_{r\in \left( 0,1\right) }|\Delta _{4}\left(
r\right) |$, so the former is better than the latter. And, the last
inequality in (\ref{A_5>4>2>1>4}) also proves this assertion.

While $\mathcal{A}_{2}\left( r^{\prime }\right) $, $\mathcal{A}_{3}\left(
r^{\prime }\right) $ and $\mathcal{A}_{5}\left( r^{\prime }\right) $($%
=S_{11/4,7/4}\left( 1,r^{\prime }\right) $) are $6$-order lower
approximations. In view of $\varepsilon _{6}^{\left( 5\right) }<\varepsilon
_{6}^{\left( 3\right) }<\varepsilon _{6}^{\left( 2\right) }$ and%
\begin{equation*}
\max_{r\in \left( 0,1\right) }|\Delta _{5}\left( r\right) |\ll \min \left(
\max_{r\in \left( 0,1\right) }|\Delta _{2}\left( r\right) |,\max_{r\in
\left( 0,1\right) }|\Delta _{3}\left( r\right) |\right) ,
\end{equation*}%
we claim that the accuracy of $\mathcal{A}_{5}\left( r^{\prime }\right) $ ($%
=S_{11/4,7/4}\left( 1,r^{\prime }\right) $) is far superior to $\mathcal{A}%
_{2}\left( r^{\prime }\right) $ and $\mathcal{A}_{3}\left( r^{\prime
}\right) $. And, the first and second inequalities in (\ref{A_5>4>2>1>4})
confirm similarly prove this claim.

To sum up, our lower approximation $\mathcal{A}_{5}\left( r^{\prime }\right) 
$($=S_{11/4,7/4}\left( 1,r^{\prime }\right) $) for $\left( 2/\pi \right)
E\left( r\right) $ is the best of all five ones listed in Table 1.
\end{remark}

\subsection{For some upper approximations}

In Tables 2, the values in the third column are derived by expanding in
power series, while $\max_{r\in \left( 0,1\right) }|\Delta _{i}\left(
r\right) |$ for $i=7$, $8$ in the fourth column are from \cite[Theorem 3.2]%
{Chu-CMA-63-2012}, (\ref{Bi-E-S_2}), respectively.

\begin{equation*}
\begin{array}{c}
\text{Table 2: \ The upper approximations} \\ 
\begin{tabular}{|c|c|c|c|}
\hline
$\mathcal{A}_{i}\left( r^{\prime }\right) $ & Expressions & $\delta
_{0}^{\left( i\right) }:=\varepsilon _{n_{0}}^{\left( i\right) }r^{2n_{0}}$
& $\max_{r\in \left( 0,1\right) }|\Delta \left( r\right) |$ \\ \hline
\multicolumn{1}{|l|}{$\mathcal{A}_{6}\left( r^{\prime }\right) $} & 
\multicolumn{1}{|l|}{$\mathcal{L}_{1/4}\left( 1,r^{\prime }\right) $} & 
\multicolumn{1}{|l|}{$-\frac{1}{2^{12}}r^{8}$} & \multicolumn{1}{|l|}{$\geq
1-\frac{2}{\pi }\approx 0.36338$} \\ \hline
\multicolumn{1}{|l|}{$\mathcal{A}_{7}\left( r^{\prime }\right) $} & 
\multicolumn{1}{|l|}{$\frac{18A\left( 1,r^{\prime }\right) -5G\left(
1,r^{\prime }\right) +3S\left( 1,r^{\prime }\right) }{16}$} & 
\multicolumn{1}{|l|}{$-\frac{7}{2^{20}}r^{12}$} & \multicolumn{1}{|l|}{$=%
\frac{18+3\sqrt{2}}{32}-\frac{2}{\pi }\approx 0.058463$} \\ \hline
\multicolumn{1}{|l|}{$\mathcal{A}_{8}\left( r^{\prime }\right) $} & 
\multicolumn{1}{|l|}{$S_{5/2,2}\left( 1,r^{\prime }\right) $} & 
\multicolumn{1}{|l|}{$-\frac{3}{5\times 2^{14}}r^{8}$} & 
\multicolumn{1}{|l|}{$=\frac{16}{25}-\frac{2}{\pi }\approx 0.0033802$} \\ 
\hline
\end{tabular}%
\end{array}%
\end{equation*}

Further we have

\begin{lemma}
\label{A-786}Let $\mathcal{A}_{i}\left( x\right) $ ($i=6,7,8$) be defined on 
$\left( 0,\infty \right) $ by (\ref{E<L_1/4}), (\ref{Ch-2}), (\ref{S_5/2,2}%
), respectively (see also Table 2). Then (i) the inequality%
\begin{equation}
\max \left( \mathcal{A}_{7}\left( x\right) ,\mathcal{A}_{8}\left( x\right)
\right) <\mathcal{A}_{6}\left( x\right)  \label{A_7,8<6}
\end{equation}%
hold for $x>0$ with $x\neq 1$; (ii) there is a $x_{0}\in \left( 0,1\right) $
such that $\mathcal{A}_{8}\left( x\right) >\mathcal{A}_{7}\left( x\right) $
for $x\in \left( 0,x_{0}\right) $ and $\mathcal{A}_{8}\left( x\right) <%
\mathcal{A}_{7}\left( x\right) $ for $x\in \left( x_{0},1\right) $.
\end{lemma}

\begin{proof}
(i) It has been proven in \cite[Lemma 5.1]{Chu-CMA-63-2012} that $\mathcal{A}%
_{7}\left( x\right) <\mathcal{A}_{6}\left( x\right) $. To prove $\mathcal{A}%
_{8}\left( x\right) <\mathcal{A}_{6}\left( x\right) $, it suffices to prove
that $\mathcal{A}_{8}\left( x^{4}\right) -\mathcal{A}_{6}\left( x^{4}\right)
<0$. Straightforward calculation gives%
\begin{eqnarray*}
\mathcal{A}_{8}\left( x^{4}\right) -\mathcal{A}_{6}\left( x^{4}\right) &=&%
\frac{16\left( x^{10}-1\right) ^{2}}{25\left( x^{8}-1\right) ^{2}}-\frac{%
x^{5}+1}{x+1}=-\frac{\left( x-1\right) ^{6}\left( x+1\right) \left(
x^{5}+1\right) }{25\left( x^{8}-1\right) ^{2}} \\
&&\times \left(
9x^{8}+20x^{7}+44x^{6}+60x^{5}+74x^{4}+60x^{3}+44x^{2}+20x+9\right) ,
\end{eqnarray*}%
which is obviously negative for $x>0$ with $x\neq 1$.

(ii) We have%
\begin{eqnarray}
\mathcal{A}_{8}\left( x^{2}\right) -\mathcal{A}_{7}\left( x^{2}\right) &=&%
\frac{16\left( x^{5}-1\right) ^{2}}{25\left( x^{4}-1\right) ^{2}}-\left( 
\frac{9}{8}\frac{1+x^{2}}{2}-\frac{5}{16}x+\frac{3}{16}\sqrt{\frac{1+x^{4}}{2%
}}\right)  \notag \\
&:&=\frac{h_{2}\left( x\right) -h_{3}\left( x\right) }{400}=\frac{1}{400}%
\frac{h_{2}^{2}\left( x\right) -h_{3}^{2}\left( x\right) }{h_{2}\left(
x\right) +h_{3}\left( x\right) },  \label{S2-Ch}
\end{eqnarray}%
where%
\begin{eqnarray*}
h_{2}\left( x\right) &=&\frac{%
31x^{8}+187x^{7}+118x^{6}+49x^{5}+430x^{4}+49x^{3}+118x^{2}+187x+31}{\left(
x^{2}+1\right) ^{2}\left( x+1\right) ^{2}}, \\
h_{3}\left( x\right) &=&75\sqrt{\frac{1+x^{4}}{2}}.
\end{eqnarray*}

Making a change of variable $u=x+1/x$ yields%
\begin{equation*}
h_{2}\left( x\right) =x\frac{31u^{4}+187u^{3}-6u^{2}-512u+256}{u^{2}\left(
u+2\right) }\text{ \ and \ }h_{3}\left( x\right) =75x\sqrt{\frac{u^{2}-2}{2}}%
,
\end{equation*}%
and therefore,%
\begin{equation}
h_{2}^{2}\left( x\right) -h_{3}^{2}\left( x\right) =-\frac{x^{2}}{2}\frac{%
\left( u-2\right) ^{2}}{u^{4}\left( u+2\right) ^{2}}h_{4}\left( u\right) ,
\label{h2-h3}
\end{equation}%
where%
\begin{equation*}
h_{4}\left( u\right)
=3703u^{6}+14\,124u^{5}-16\,260u^{4}-98\,560u^{3}-23\,040u^{2}+98\,304u-32%
\,768.
\end{equation*}%
Since $u=x+1/x\geq 2$, replacing $u$ by $v+2$ leads us to%
\begin{equation*}
h_{4}\left( v+2\right)
=3703v^{6}+58\,560v^{5}+347\,160v^{4}+928\,800v^{3}+1014\,000v^{2}+144%
\,000v-288\,000,
\end{equation*}%
where $v>0$.

It follows from Lemma \ref{L-P-zp} that the polynomial $h_{4}\left(
v+2\right) $ has a unique zero point $v_{1}\in \left( 0,\infty \right) $
such that $h_{4}\left( v+2\right) <0$ for $v\in \left( 0,v_{1}\right) $ and $%
h_{4}\left( v+2\right) >0$ for $v\in \left( v_{1},\infty \right) $. Numeric
computation gives $v_{1}\in \left( 0.399\,475162,0.399\,475163\right) $.
This together with (\ref{h2-h3}) and (\ref{S2-Ch}) reveals that%
\begin{equation*}
\mathcal{A}_{8}\left( x^{2}\right) -\mathcal{A}_{7}\left( x^{2}\right)
\left\{ 
\begin{array}{cc}
>0 & \text{for }x\in \left( 0,x_{1}\right) , \\ 
<0 & \text{for }x\in \left( x_{1},1\right) ,%
\end{array}%
\right.
\end{equation*}%
where $x_{1}+1/x_{1}=2+v_{1}$, that is, $x_{1}=\left( v_{1}+2-\sqrt{%
v_{1}\left( v_{1}+4\right) }\right) /2\approx 0.53689$, which proves the
second assertion, where $x_{0}=x_{1}^{2}\approx 0.28825$.

This completes the proof.
\end{proof}

\begin{remark}
From Table 2, as upper approximations, $\mathcal{A}_{6}\left( r^{\prime
}\right) $($=\mathcal{L}_{1/4}\left( 1,r^{\prime }\right) $) and $\mathcal{A}%
_{8}\left( r^{\prime }\right) $($=S_{5/2,2}\left( 1,r^{\prime }\right) $)
have $8$-order accuracy, but the facts $|\varepsilon _{4}^{\left( 8\right)
}|<|\varepsilon _{4}^{\left( 6\right) }|$ and $\max_{r\in \left( 0,1\right)
}|\Delta _{8}\left( r\right) |<\max_{r\in \left( 0,1\right) }|\Delta
_{6}\left( r\right) |$ together with the inequality (\ref{A_7,8<6}) show
that the accuracy of $\mathcal{A}_{8}\left( r^{\prime }\right) $ is higher
than $\mathcal{A}_{6}\left( r^{\prime }\right) $.

The upper approximation $\mathcal{A}_{7}\left( r^{\prime }\right) $ has $6$%
-order accuracy, but its maximum absolute error is greater than our
approximation $\mathcal{A}_{8}\left( r^{\prime }\right) $'s. By the second
assertion of Lemma \ref{A-786}, it is seen that there is a unique $r_{0}\in
\left( 0,1\right) $ such that $\left( 2/\pi \right) E\left( r\right) <%
\mathcal{A}_{7}\left( r^{\prime }\right) <\mathcal{A}_{8}\left( r^{\prime
}\right) $ for $r\in \left( 0,r_{0}\right) $ and $\left( 2/\pi \right)
E\left( r\right) <\mathcal{A}_{8}\left( r^{\prime }\right) <\mathcal{A}%
_{7}\left( r^{\prime }\right) $ for $r\in \left( r_{0},1\right) $, where $%
r_{0}=\sqrt{1-x_{0}^{2}}\approx 0.95756$.
\end{remark}

\begin{remark}
Not only that but our approximation $S_{11/4,7/4}\left( 1,r^{\prime }\right) 
$ and $S_{5/2,2}\left( 1,r^{\prime }\right) $ have very small absolute
relative errors. Exactly, from (\ref{MI-7/4-2}) and (\ref{MI-2-2}) it is
derived that%
\begin{eqnarray*}
\left\vert \frac{\left( 2/\pi \right) E\left( r\right) -S_{11/4,7/4}\left(
1,r^{\prime }\right) }{\left( 2/\pi \right) E\left( r\right) }\right\vert
&<&1-\frac{7\pi }{22}\approx 0.00040234, \\
\left\vert \frac{\left( 2/\pi \right) E\left( r\right) -S_{5/2,2}\left(
1,r^{\prime }\right) }{\left( 2/\pi \right) E\left( r\right) }\right\vert &<&%
\frac{8\pi }{25}-1\approx 0.0053096.
\end{eqnarray*}%
Furthermore, by Corollary \ref{MC-E/S_p-mon.}, the absolute relative error
function%
\begin{equation*}
\mathcal{E}_{p}\left( r\right) =\left\vert \frac{\left( 2/\pi \right)
E\left( r\right) -S_{9/2-p,p}\left( 1,r^{\prime }\right) }{\left( 2/\pi
\right) E\left( r\right) }\right\vert
\end{equation*}%
is strictly increasing in $r$ from $\left( 0,1\right) $ onto $\left(
0,\left\vert 1-\pi \theta _{p}/2\right\vert \right) $ for $p\in (0,7/4]\cup %
\left[ 5/2,9/4\right] $. Also, by the Remark \ref{P3-C}, the maximum
absolute relative error $\max_{r\in \left( 0,1\right) }\mathcal{E}_{p}\left(
r\right) =\left\vert 1-\pi \theta _{p}/2\right\vert $ is strictly decreasing
in $p$ on $(0,7/4]$ and strictly increasing on $[2,9/4]$.
\end{remark}

Lastly, we close this paper by proposing a conjecture as follows.

\begin{conjecture}
Let $p_{0}\approx 1.763135$ denote the unique solution of the equation $%
\theta _{p}=2/\pi $, that is,%
\begin{equation*}
\left( \frac{9/2-p}{p}\right) ^{1/\left( 2p-9/2\right) }=\frac{2}{\pi }
\end{equation*}%
on $(0,9/4]$. Then there is a $r_{0}\in \left( 0,1\right) $ such that the
function%
\begin{equation*}
H\left( r\right) =\frac{1-\left( 2/\pi \right) E\left( r\right) }{%
1-S_{9/2-p_{0},p_{0}}\left( 1,r^{\prime }\right) }
\end{equation*}%
is strictly increasing on $\left( 0,r_{0}\right) $ and strictly decreasing
on $\left( r_{0},1\right) $. Consequently, it holds that%
\begin{equation*}
\frac{2}{\pi }E\left( r\right) <S_{9/2-p_{0},p_{0}}\left( 1,r^{\prime
}\right)
\end{equation*}%
for $r\in \left( 0,1\right) $ with the best constant $p_{0}$.
\end{conjecture}

\end{document}